\documentclass[10pt,leqno]{amsart}
\usepackage{amsfonts,amssymb, latexsym}
\usepackage{amsmath,amscd}
\usepackage{cases}

\DeclareMathOperator{\im}{Im} 
\DeclareMathOperator{\re}{Re} 

\DeclareMathOperator{\comp}{comp}
\DeclareMathOperator{\loc}{loc}

\usepackage{amssymb}
\usepackage{mathrsfs}
\usepackage{graphicx}
\usepackage{color}
\usepackage[all]{xy}
\usepackage{bigstrut}
\usepackage{mathtools}

\newtheorem{exttheo}{Theorem}
\newtheorem{theorem}{Theorem}[section]
\newtheorem{proposition}[theorem]{Proposition}
\newtheorem{lemma}[theorem]{Lemma}

\newtheorem{corollary}[theorem]{Corollary}

\newtheorem{remark}[theorem]{Remark}

\theoremstyle{definition}%
\newtheorem{definition}[theorem]{Definition}
\newtheorem{assumption}{Assumption}
\newtheorem{notation}[theorem]{Notation}
\theoremstyle{remark}

\def\pasdegrille{\let\grille = \pasgrille}

\def\aat#1#2#3{
\divide \dimen1 by 48 \dimen3=\dimen1 \multiply \dimen1 by #1
\advance \dimen1 by -\dimen3 \divide \dimen1 by 101 \multiply
\dimen1 by 100 \divide \dimen2 by \count11 \multiply \dimen2 by #2
\setbox0=\hbox{#3}\ht0=0pt\dp0=0pt
  \rlap{\kern\dimen1 \vbox to0pt{\kern-\dimen2\box0\vss}}\dimen1= \wd1
\dimen2=\ht1}
\def\pasgrille{
\count12= \dimen1 \divide \count12 by 50 \divide \dimen2 by \count12
\count11 =\dimen2 \ \divide \dimen1 by 48
\setlength{\unitlength}{\dimen1} \smash{\rlap{\ }} \dimen1= \wd1
\dimen2=\ht1 }
\def\grille{
\count12= \dimen1 \divide \count12 by 50 \divide \dimen2 by \count12
\count11 =\dimen2 \ \divide \dimen1 by 48
\setlength{\unitlength}{\dimen1}
\smash{\rlap{\graphpaper[1](0,0)(50, \count11)}} \dimen1= \wd1
\dimen2=\ht1 }

\pasdegrille

\numberwithin{equation}{section}

\newcommand{\R}{\mathbb{R}}
\newcommand{\C}{\mathbb{C}}

\newcommand{\Z}{\mathbb Z}

\newcommand{\supp}{\operatorname{Supp}}

\newcommand{\beq}{\begin{equation}}
\newcommand{\eeq}{\end{equation}}
\newcommand{\ben}{\begin{eqnarray}}
\newcommand{\een}{\end{eqnarray}}
\newcommand{\beno}{\begin{eqnarray*}}
\newcommand{\eeno}{\end{eqnarray*}}

\newcommand{\indic}{1\!\!1}

\def\eps{\varepsilon}

\newcommand{\la}{\langle}
\newcommand{\ra}{\rangle}

\newcommand{\AAA}{\mathcal{A}}

\newcommand{\HHH}{\mathcal{H}}

\newcommand{\PPP}{\mathcal{P}}

\newcommand{\uv}{\underline{v}}

\title[Dispersive estimates with potential in exterior domains]{Dispersive estimates for wave and Schr\"odinger equations with a potential in non-trapping exterior domains}
\author{Thomas Duyckaerts${}^1$}
\address{Thomas Duyckaerts, LAGA, Institut Galil\'ee, Universit\'e Paris 13
99, avenue Jean-Baptiste Cl\'ement,
93430 - Villetaneuse, France and
DMA, \'Ecole Normale Sup\'erieure, Universit\'e PSL, CNRS, 75005 Paris, France}
\email{duyckaer@math.univ-paris13.fr}

\author{Jianwei Yang${}^2$}
\address{Jianwei Yang, Department of Mathematics, Beijing Institute of Technology, Beijing 100081, P. R. China}
\email{jw-urbain.yang@bit.edu.cn}

\thanks{$^1$LAGA (UMR 7539), Universit\'e Paris Sorbonne Paris Nord DMA, \'Ecole Normale Sup\'erieure, Universit\'e PSL. Partially supported by the Labex MME-DII and the \emph{Institut Universitaire de France}}
\thanks{$^2$Department of Mathematics, Beijing Institute of Technology. Supported by NSFC grant No. 12371239 and Research fund program for young scholars of Beijing Institute of Technology.}


\begin{document}

\begin{abstract}
 We prove resolvent estimates for a Schr\"odinger operator with a short-range potential outside an obstacle with Dirichlet boundary conditions. As a consequence, we deduce integrability of the local energy for the wave equation, and smoothing effect for the Schr\"odinger equation. Finally, for both equations, we prove that local Strichartz estimates for the free equation outside an obstacle imply global Strichartz estimates with a short-range potential outside the same obstacle. The estimates are all global in time, after projection on the continuous spectrum of the operator.
\end{abstract}

\maketitle

\section{Introduction}

Let $N\ge 3$ and $K\subset\R^N$ be a compact set with $C^2$ boundary $\partial K$, and $V$ a real-valued, continuous potential on $\R^N$ that decays to $0$ at infinity. 

The aim of this note is to study global Strichartz estimates for the linear wave equation
\begin{equation}
 \label{Wave_intro}
 \left\{\begin{aligned}
 \partial_t^2u-\Delta u+Vu&=f, &&(t,x)\in \R\times \Omega\\
 u(t,x)&=0, &&(t,x)\in \R\times \partial\Omega
 \end{aligned}
 \right.
 \end{equation} 
and the Schr\"odinger equation 
\begin{equation}
\label{LS_intro}
 \left\{\begin{aligned}
i\partial_t\varphi+\Delta \varphi-V\varphi&=\psi,&&(t,x)\in \R\times \Omega\\
\varphi(t,x)&=0,&&(t,x)\in \R\times \partial\Omega
 \end{aligned}
 \right.
\end{equation}
on $\Omega$, with potential $V$ and Dirichlet boundary conditions. One of our motivation is the study of the linearized equation around a stationary state of the nonlinear wave equation outside an obstacle, see \cite{DuyckaertsYang23Pb}. We will consider finite energy solutions of \eqref{Wave_intro}, that is solutions with initial data 
\begin{equation}
\label{Wave_ID}
(u,\partial_tu)_{\restriction t=0}=(u_0,u_1)\in \mathcal{H}:= \dot{H}^1_0(\Omega)\times L^2(\Omega),
\end{equation} 
(where $\dot{H}^1_0(\Omega)$ is the closure of $C^{\infty}_0(\Omega)$ for the norm defined by $\|f\|^2_{\dot{H}^1_0(\Omega)}=\int_{\Omega} |\nabla f|^2$),
and finite mass solutions of \eqref{LS_intro}:
\begin{equation}
 \label{LS_ID}
 \varphi_{\restriction t=0}=\varphi_0\in L^2(\Omega).
\end{equation}
We refer to Appendix \ref{AA:lower} for more general initial conditions than \eqref{Wave_ID} for equation \eqref{Wave_intro}.

We first recall Strichartz estimates for the free equations on the Euclidean space (denoting $L^pL^q=L^p(\R,L^q(\Omega))$):
\begin{exttheo}
\label{T:free_Strichartz}
Assume $\Omega=\R^N$, $V=0$. 
\begin{itemize}
 \item 
Assume $(p,q)$ satisfies $2\leq p\leq \infty$, $2\leq q<\infty$ and 
\begin{equation}
\label{admissible_wave}
 \frac{1}{p}+\frac{N}{q}=\frac{N}{2}-1.
\end{equation}
Then for any solution $u,f$ of \eqref{Wave_intro}, \eqref{Wave_ID} with $f\in L^1L^2$ one has
\begin{gather}
\label{StrichartzWave}
 \|u\|_{L^pL^q}\lesssim \|(u_0,u_1)\|_{\mathcal{H}}+\|f\|_{L^1L^2},
\end{gather}
\item Assume that for $j=1,2$, $(p,q)=(p_j,q_j)$ satisfy $2\leq p\leq \infty$, $2\leq q<\infty$ and 
\begin{equation}
 \label{admissible_LS}
 \frac{2}{p}+\frac{N}{q}=\frac{N}{2}.
\end{equation} 
Then for any solution $\varphi$, $\psi$ of \eqref{LS_intro}, \eqref{LS_ID} with $\psi\in L^{p_2'}L^{q_2'}$, one has
\begin{gather}
\label{StrichartzLS}
 \|\varphi\|_{L^{p_1}L^{q_1}}\lesssim \|\varphi_0\|_{L^2}+\|\psi\|_{L^{p_2'}L^{q_2'}}.
\end{gather}
\end{itemize}
\end{exttheo}

Strichartz estimates for the free wave and Schr\"odinger equations were introduced in \cite{St77a} and generalized by several authors: see \cite{GiVe85}, \cite{Yajima87}, \cite{CaWe88} for the Schr\"odinger equation, \cite{Kapitanski90}, \cite{GiVe95} and \cite{LiSo95} for the wave equation and \cite{KeTa98} for the endpoints $p=2$ (respectively $p_1=2$, $p_2=2$).

We note that for the wave equation, when initial data are in $\HHH$ as in Theorem \ref{T:free_Strichartz}, the usual admissibility condition $\frac{2}{p}+\frac{N-1}{q}\leq \frac{N-1}{2}$ (see e.g. \cite[Definition 1.1]{KeTa98}) is a consequence of \eqref{admissible_wave}.

When $\Omega=\R^N$ and $V\neq 0$, it is easy to deduce (at least if $V$ is regular enough and decays fast enough at infinity) local-in-time Strichartz estimates for \eqref{Wave_intro}, \eqref{LS_intro} from Theorem \ref{T:free_Strichartz} and H\"older's inequality. With an appropriate smallness assumption on the potential, one obtains global estimates. These global estimates fail for more general potential due to the possible existence of negative eigenvalues, and/or of an eigenvalue or a resonance at $0$ for the operator $-\Delta +V$ (see Definition \ref{def:resonance} below). Projecting on the continuous spectrum or assuming that the operator is positive to avoid the negative eigenvalues, and assuming that $0$ is not an eigenvalue or a resonance, global Strichartz estimates were proved with an almost optimal decay assumption on the potential in various contexts, see e.g. \cite{RodnianskiSchlag04},  \cite{MetcalfeTataru12}, \cite{Beceanu11}, \cite{BoucletMizutani18}, \cite{DAncona20} and reference therein. The optimal decay assumption in these articles is roughly $|V(x)|\lesssim |x|^{-2}$ at infinity or $V\in L^{n/2}$ (short-range potential).

Several authors have studied the case $V=0$, $K\neq \emptyset$, proving local or global Strichartz estimates for both equations \eqref{Wave_intro} and \eqref{LS_intro} under a suitable geometric assumption on $K$. Unlike the case where $K=\emptyset$, $V\neq 0$, even the local Strichartz estimates are far from being trivial. When the obstacle is strictly convex, all Strichartz estimates of Theorem \ref{T:free_Strichartz} are valid (see Remarks \ref{R:pqwave}, \ref{R:pqLS} below).


A weaker geometric assumption on $K$ is that it is a nontrapping obstacle, i.e. that every ray of geometric optics on $\Omega$ exit from every compact subset of $\overline{\Omega}$. It was proved by Smith and Sogge \cite{SmithSogge00} (for odd $N$), Burq \cite{Burq03} and Metcalfe \cite{Metcalfe04}, Metcalfe and Tataru \cite{MetcalfeTataru08}  with this non-trapping assumption, that a local Strichartz estimate implies the corresponding global one.
Let us mention however that the non-trapping assumption is not a necessary condition for the validity of Strichartz estimates (see e.g. \cite{Lafontaine22}).

In this work, we will generalize the result of \cite{SmithSogge00}, \cite{Burq03} and \cite{Metcalfe04},   proving that a local Strichartz estimate for one of the equations above with $V=0$, outside a compact obstacle, implies the corresponding global one, with a (possibly large and negative) potential $V$. We will make an almost optimal short-range decay assumption at infinity for $V$.

We next detail our assumptions on $K$ and $V$. Let $-\Delta_D$ be the Dirichlet Laplacian on $\Omega=\R^N\setminus K$, with domain $H_0^1(\Omega)\cap H^2(\Omega)$.
\begin{assumption}
\label{As:K}
$K\subset \R^N$, with $C^2$ boundary, and $\forall \chi\in C_0^\infty(\mathbb{R}^N)$, 
\begin{equation}
\label{non-trapping}
\max_{{\rm Im}\lambda\neq 0}
\sqrt{1+|\lambda|}\;\| \chi (-\Delta_D-\lambda)^{-1}\chi\|_{L^2(\Omega)\to L^2(\Omega)}<+\infty.
\end{equation}
 \end{assumption}
The estimate \eqref{non-trapping} should be seen as a non-trapping assumption on $K$. 
It holds when $K$ is a non-trapping, smooth obstacle: see \cite[Remark 2.6]{Burq03} 
and reference therein.
\begin{assumption}
\label{As:V}
The potential $V\in C(\overline{\Omega})$ is a real-valued function such that 
\begin{equation}
\label{decayV}
\exists\; \alpha>2, C>0,\;\forall x\in \Omega,\quad |V(x)|\le C(1+|x|)^{-\alpha}.
\end{equation}  
\end{assumption}

Let $L_V=-\Delta+V$ on $L^2(\Omega)$, subject to Dirichlet boundary condition, defined as an unbounded self-adjoint operator on $L^2(\Omega)$ with the same domain $H^2(\Omega)\cap H^1_0(\Omega)$ as $-\Delta_D$.

The general theory on the spectrum   of $L_V$ is classical.
Since $V$ is a Kato potential by \eqref{decayV}, it is relatively $-\Delta_D-$compact and hence the essential spectrum of $L_V$ coincides with that of $-\Delta_D $ by Weyl's theorem. By the decay property \eqref{decayV}, the operator $L_V$ has no  positive eigenvalues embedded into $(0,+\infty)$, see \cite{IoJe03}. 
By a compactness argument,  $L_V$ has a finite number of negative eigenvalues with finite multiplicities \cite{DuKeMe16a}. Consider an orthonormal family $\{Y_j\}_{j=1,\ldots, J}$ of real-valued eigenfunctions corresponding to all these negative eigenvalues (counted with multiplicities). Then,  $\{Y_j(x)\}_{j=1,\ldots,J}$ are $C^2$ and decreasing exponentially as $|x|\to+\infty$.

We define $P_c$ the orthogonal projection to the subspace associated to the continuous spectrum of $L_V$, for the $L^2$ scalar product: 
\begin{equation}
\label{formulaPc}
P_cf=f-\sum_{j=1}^{J}{\rm Re}\Bigl(\int_{\Omega}Y_j f\Bigr)Y_j.
\end{equation}
Note that \eqref{formulaPc} makes sense for $f\in L^{2}(\Omega)$ and for $f\in \dot{H}^1_0(\Omega)$ due to the decay and smoothness properties of the eigenfunctions $Y_j$.

For $(u_0,u_1)\in\mathcal{H}$, define $\mathcal{P}_c(u_0,u_1)=(P_cu_0,P_c u_1)$.


We will make an additional assumption to exclude a singular behaviour of $L_V$ at the bottom of its spectrum $\lambda=0$.
\begin{definition}
	\label{def:resonance} 
	We say that $v$ is a zero-energy state for $L_V$, if there exists a 
	nonzero $v\in H^2_{\mathrm{loc}}(\overline{\Omega})$ such that $L_V v=0$, $v_{\restriction_{\partial\Omega}}=0$ and
	$\langle x \rangle^{-1} v\in L^2(\Omega)$. We say that zero is a resonance of $L_V$ when $L_V$ has a zero-energy state which is not an eigenfunction.
\end{definition}
\begin{assumption}
 \label{As:zero}
 Zero is not an eigenvalue or a resonance of $L_V$.
\end{assumption}
\begin{remark}
\label{R:resonance}
 The choice of the condition $\langle x \rangle^{-1} v\in L^2(\Omega)$ in the definition of a zero-energy state is arbitrary. Indeed, there exist several definitions of zero-resonances in the literature (see for examples  \cite{RodnianskiTao15,DAncona20} in the case $K=\emptyset$). The difference between these definitions is the weighted $L^2$ space which is chosen. By Proposition \ref{P:decay_resonance}, proved in the appendix as a direct consequence of a Stein-Weiss fractional integral estimate \cite{SW58}, these definitions are equivalent to the definition proposed here. 
 \end{remark}

\begin{remark}
\label{R:example}
	When $N\ge 5$, by Proposition \ref{P:decay_resonance} in the appendix, if $V\in C(\overline{\Omega})$ is real-valued and satisfies \eqref{decayV}, zero is never a resonance for $L_V$. In the cases $N=3$ and $N=4$, there exists potential $V$ and obstacles $K$ such that zero is a resonance for $L_V$. Indeed, when $K=\emptyset$, or $K$ is the  Euclidean ball centered at $0$  with radius $\sqrt{N(N-2)}$, 
	$$W=\left( 1+\frac{|x|^2}{N(N-2)} \right)^{1-\frac{N}{2}},$$
	and 
	$$V=-\frac{N+2}{N-2} W^{\frac{4}{N-2}},\quad v=x\cdot \nabla W+\frac{N-2}{2}W$$
	then 
	$$v_{\restriction \partial\Omega}=0,\quad L_Vv=0,\quad \forall \delta >0, \;\la x\ra^{\frac{N}{2}-2-\delta}v\in L^2(\Omega),$$
	and
	$$\la x\ra^{\frac{N}{2}-2} v\notin L^2(\Omega),$$
	so that in particular $v\notin L^2(\Omega)$ if $N\in \{3,4\}$.
\end{remark}
We first state our main result in the case of the wave equation. We will assume that the pair $(p,q)$ is an admissible Strichartz pair for the free wave equation on $\R^N$, and that local Strichartz estimates are also available for the wave equation without potential on $\Omega$:
\begin{assumption}
\label{As:pqwave}
 The pair $(p,q)$ satisfies \eqref{admissible_wave}, with $p\in(2,\infty]$, $q\in [2,\infty)$. Furthermore there exist constants $R>0$ (such that $K\subset \{|x|\leq R-1\}$), $C>0$ such that for any solution of \eqref{Wave_intro}, \eqref{Wave_ID}  with $V=0$,  $f=0$, and $\supp(u_0,u_1)\subset \{|x|\leq R\}$ one has
		\begin{align*}
		\|u\|_{L^p_t([0,1]; L_x^q(\Omega))}+\|(u,\partial_t u)\|_{L^\infty_t([0,1];\HHH)}\le C\|(u_0,u_1)\|_{\HHH}.
		\end{align*}
\end{assumption}


\begin{remark}
\label{R:pqwave}
Due to the boundary, it is rather delicate to determine the exact range of exponents $p,q$ that satisfies Assumption \ref{As:pqwave}. It follows from 
\cite{SmithSogge95} that if $K$ is smooth and strictly convex, every pair $(p,q)$ such that \eqref{admissible_wave} holds and $q<\infty$ also satisfies Assumption \ref{As:pqwave}. By \cite{BSS} if $K$ is a general smooth obstacle, $(p,q)$ satisfy Assumption \ref{As:pqwave} if \eqref{admissible_wave} holds and
\begin{equation}
	\label{condition_BSS}
	\begin{cases}
	q\le \frac{2(N-1)}{N-3}&\text{ if }N\geq 4\\
	q\le 14&\text{ if }N=3.
	\end{cases}
	\end{equation}
Global Strichartz estimates when $K$ is a union of two strictly convex obstacles are proved in \cite{Lafontaine22}. Note that in this case $K$ does not satisfies Assumption \ref{As:K}. Assumption \ref{As:K} is not sufficient neither, at least if one wants the full range of admissible exponents appearing in Theorem \ref{T:free_Strichartz} (see the counter-examples in \cite{Ivanovici1}, \cite{Ivanovici2}).
\end{remark}
\begin{remark}
 When $\Omega$ is a ball, if Assumption \ref{As:zero} holds for the operator $L_V$ restricted to radial functions, then all the results below hold when $(u_0,u_1)$ and $f$ are radial. Note that in this case Assumption \ref{As:pqwave} holds for all pair $(p,q)$ satisfying \eqref{admissible_wave}.
\end{remark}

\begin{theorem}
	\label{T:wave}
	Let $N\geq 3$.
	Let $\Omega$ and $V$ satisfy Assumptions \ref{As:K}, \ref{As:V}, \ref{As:zero}. Let $(p,q)$ such that Assumption \ref{As:pqwave} holds. Then, there is $C>0$ such that for all solutions $u$ of \eqref{Wave_intro}, \eqref{Wave_ID} with $f\in  L^{1} L^{2}$, one has 
	\begin{equation}
	\label{eq:strichartz-nonedpt}
	\|P_c u\|_{L^pL^q}+\|\PPP_c(u,\partial_tu)\|_{L^{\infty}(\R,\HHH)}\le C\Bigl(\|\PPP_c(u_0,u_1)\|_{\mathcal{H}}+\|P_c f\|_{L^{1}L^{2}}\Bigr).
	\end{equation}
\end{theorem} 
We refer to Appendix \ref{AA:lower} for Strichartz estimates at other levels of regularity.

We next state our result for the Schr\"odinger equation. As in the case of the wave equation, we will assume that the Strichartz exponents appearing in the inequality are admissible for the free equation on $\R^N$, and that the local Strichartz estimates are available for the Schr\"odinger equation without potential on $\Omega$. 

\begin{assumption}
\label{As:pqLS}
The pair $(p,q)$ satisfies \eqref{admissible_LS} with $p\in (2,\infty]$, $q\in [2,\infty)$. Furthermore,  there exist a constant  $C>0$, and a function $\chi\in C_0^{\infty}(\R^N)$ such that $\chi=1$ in a neighborhood of $K$ and for any $\varphi_0\in L^2(\Omega)$,
\begin{equation}
\label{local_Strichartz_LS}
\|\chi e^{it\Delta_D}\varphi_0\|_{L^{p}([0,1],L^{q})}\leq C\|\varphi_0\|_{L^2}.
\end{equation}
\end{assumption}
\begin{remark}
\label{R:pqLS}
By \cite{IvanoviciNLS}, if $K$ is a strictly convex obstacle, every pair $(p,q)$ such that $p>2$ and \eqref{admissible_LS} holds satisfies Assumption \ref{As:pqLS}. Indeed Ivanovici proves in \cite{IvanoviciNLS} a global Strichartz inequality.\footnote{Note that Theorem \ref{T:LS} implies that for the Schr\"odinger equation without potential outside a non-trapping obstacle, the local Strichartz inequality \eqref{local_Strichartz_LS} is equivalent to a global one.} Assumption \ref{As:pqLS} is also valid for the full set of Strichartz exponents where $K$ is the union of several strictly convex obstacles (for which however, Assumption \ref{As:K} does not hold). See \cite{Lafontaine18P}.
\end{remark}

\begin{theorem}
 \label{T:LS}
 	Let $\Omega$ and $V$ satisfy Assumptions \ref{As:K}, \ref{As:V}, \ref{As:zero}. Let $(p_1,q_1)$, $(p_2,q_2)$ such that Assumption \ref{As:pqLS} holds for both pairs. Then, there is $C>0$ such that for all solutions $\varphi$ of \eqref{LS_intro}, \eqref{LS_ID} with $\psi\in  L^{p_2'}L^{q_2'}$, one has
	\begin{equation}
	\label{eq:strichartzLS}
	\|P_c \varphi\|_{L^{p_1}L^{q_1}}\le C\Bigl(\|P_c\varphi_0\|_{L^2}+\|P_c \psi\|_{L^{p_2'}L^{q_2'}}\Bigr).
	\end{equation}
\end{theorem}
Theorem \ref{T:LS} is exaclty the generalization of \cite[Theorem 1.4]{RodnianskiSchlag04} outside a non-trapping obstacle. See also \cite{MaMeTa08} for more general, time dependent, perturbations of the Laplace operator.

\begin{remark}
For general nontrapping obstacles, Strichartz estimates with loss of derivatives (see \cite{Anton08}) or at other levels of regularities than the one given by Assumption \ref{As:pqLS}
 (see cite \cite{BlSmSo12}) are available. We will also prove the analog of Theorem \ref{T:LS} for these general Strichartz estimates:
see Section \ref{S:LS} below, and in particular Assumption \ref{As:pqLS'} and Theorem \ref{T:LS'}.
 \end{remark}

The proofs of Theorems \ref{T:wave} and \ref{T:LS} are based on $L^2$ space-time integrability properties of solutions of equations \eqref{Wave_intro} and \eqref{LS_intro}, that we state for their own interest.

For $\delta\in [0,1]$, we will denote by $H^{\delta}_D$ the complex interpolate between $L^2(\Omega)$ and $H^{1}_0(\Omega)$, with exponent $\delta$, and $H^{-\delta}_D$ its dual with respect to the pivot space $L^2$. We note that these spaces are invariant by the Schr\"odinger flows $e^{it\Delta_D}$ and $e^{itL_V}$, since $H^1_0(\Omega)$ and $L^2(\Omega)$ are.

\begin{theorem}
 \label{T:L2}
 Let $\Omega$ and $V$ satisfy Assumptions \ref{As:K}, \ref{As:V} and \ref{As:zero}. 
 Let $\nu_1,\nu_2>1/2$ with $\nu_1+\nu_2>2$.  
\begin{description}
\item[$L^2$-integrability of the weighted energy for the wave equation]
{Assume furthermore $(u_0,u_1)=(0,0)$ or $\nu_1+\alpha>7/2$.}
For any solution $u$ of \eqref{Wave_intro}, \eqref{Wave_ID}, such that $\la x\ra^{\nu_2}P_cf\in L^2(\R\times\Omega)$, 
\begin{multline}
	\label{eq:lse}
	\lVert  \langle x\rangle^{-\nu_1}\PPP_c\vec{u}\rVert_{L^2_t(\R; \dot{H}^1_0(\Omega)\times L^2(\Omega))}\\
	\lesssim \Bigl(\lVert \PPP_c(u_0,u_1)\rVert_{\dot{ H}^1_0(\Omega)\times L^2(\Omega)}+\lVert \langle x\rangle^{\nu_2} P_c f\rVert_{L^2_{t,x}(\R\times\Omega)}\Bigr).
	\end{multline}
	In the case where $(u_0,u_1)=(0,0)$ one also has
	\begin{equation}
	 \label{eq:lse_bis}
	\lVert  \langle x\rangle^{-\nu_1}\PPP_c\vec{u}\rVert_{L^2_t(\R\times \Omega)}\\
	\lesssim \lVert \langle x\rangle^{\nu_2} P_c f\rVert_{L^2_{t,x}(\R\times\Omega)}.
	\end{equation}
\item[Smoothing effect for the Schr\"odinger equation]

For any solution $\varphi$ of \eqref{LS_intro}, \eqref{LS_ID}, with $\la x\ra^{\nu_2} P_c \psi\in L^2(\R,H^{-1/2}_D)$, 
\begin{equation}
\label{smoothing}
\left\| \la x\ra^{-\nu_1} P_c\varphi\right\|_{L^2(\R,H^{1/2}_D)}\lesssim \|P_c \varphi\|_{L^2(\Omega)}+\|\la x\ra^{\nu_2} P_c \psi\|_{L^2(\R,H_D^{-1/2})}. 
\end{equation}
 \end{description}
\end{theorem}
Estimates \eqref{eq:lse} and \eqref{smoothing} are classical when $V=0$ at the exterior of a non-trapping obstacle $K$. In this context they can be proved as a consequence of the resolvent estimate \eqref{non-trapping}. For the wave equations outside a convex obstacle, the study of local energy decay goes back to the works of Morawetz, see e.g. \cite{Morawetz61}. A weak form of \eqref{eq:lse} is proved in \cite[Section 2]{MRS77}, \cite{Burq03}. Note that in these works, the weight $\la x\ra^{-\nu_1}$ is replaced by a truncation, and the right-hand side is compactly supported. It is important for us to obtain the estimate with the weights $\la x\ra^{-\nu_1}$, $\la x\ra^{\nu_2}$ and  the conditions on $\nu_1,\nu_2$ above, to obtain Strichartz estimates for potentials satisfying the decay \eqref{decayV} in Assumption \ref{As:V}.

The global-in-time smoothing effect for the Schr\"odinger equation on the Euclidean space goes back to Kato's theory \cite{Kato66} (see also \cite{Ve88}). The inequality \eqref{smoothing} with $V=0$ outside a non-trapping obstacle is proved in \cite{BuGeTz04a} (in this article, the polynomial weight are replaced by truncations, however the same proof, together with Proposition \ref{pp:VzeroND} below, gives the exact estimate \eqref{smoothing}).

The fact that Theorem \ref{T:L2} and the Strichartz estimates for the equations without potential imply global Strichartz estimates for the equations with a potential is quite standard, see e.g. the section 4 of \cite{RodnianskiSchlag04} for the case of the Schr\"odinger equations. The idea to use \eqref{eq:lse} as a substitute of \eqref{smoothing} for the wave equation comes from the work of N. Burq \cite{Burq03}. In both cases, the key ingredients of the proof are Theorem \ref{T:L2} and an inequality combining Strichartz estimates and the inequalities of Theorem \ref{T:L2} in the case $V=0$ (see \eqref{mixed_Delta_D}, \eqref{LS_mixte}). The proof of this mixed inequality is based on a Lemma of Christ and Kiselev's  \cite{ChKi01}. With the same proof, combining Theorem \ref{T:wave} (respectively Theorem \ref{T:LS}) with the $L^2$ estimates of Theorem \ref{T:L2}, one obtains mixed inequalities combinining the Strichartz estimates and the $L^2$ integrability of the weighted energy (respectively the smoothing effect). These estimates are stated in Appendix \ref{AA:mixed}.

Theorem \ref{T:L2} follows in a standard manner from a uniform resolvent estimate \eqref{eq:alpND}. The core of the proof of Theorem \ref{T:L2} is the proof of the resolvent estimate, which we carry out in Section \ref{S:uniform}. The proof is purely perturbative (using the corresponding inequality in the case $V=0$) for large $\re \lambda$. For $\eps\leq \re \lambda <C$, the proof relies ultimately on the absence of embedded eigenvalues in the positive spectrum, proved in \cite{IoJe03}. For $\lambda$ close to $0$, we use Assumption \ref{As:zero} on the absence of zero energy state.

Theorems \ref{T:wave} and \ref{T:LS} might be known in the community,
however we were not able to find complete proofs of these results.
Let us mention the articles \cite{MetcalfeTataru12} for the wave equation, and \cite{MaMeTa08} for the Schr\"odinger equation, where dispersive estimates are obtained for perturbations of the Euclidean setting. These articles focus on (very general) variable coefficients Schr\"odinger and wave operators on $\R^N$, and although the method should adapt to the case of the exterior domain, as mentioned in \cite[Subsection 1.2.2]{MaMeTa08}, no proofs of the exact statements of Theorems \ref{T:wave} and \ref{T:LS} are given there. We also mention \cite[Theorem 3]{MetcalfeTataru08} were the $L^2$ integrability of the energy of the wave equation is obtained outside a star-shaped obstacle, for a quite general perturbation of $-\Delta$ which include operators of the form $-\Delta+V$, where $V$ satisfies a smallness assumption (precluding the existence of eigenvalues).

We finish our introduction by discussing the optimality of our result and open questions. In full generality, the decay assumption \eqref{decayV} is almost optimal. Indeed, an adaptation of \cite{Duyckaerts07a} gives a smooth positive potential, with decay $(\log |x|)|x|^{-2}$, such that any local dispersive estimate fails. Let us mention however that it is possible to prove Strichartz estimates for the Schr\"odinger equations with some long-range repulsives potential, including $V=|x|^{-\mu}$, $0<\mu<2$ (see \cite{Mizutani20}). We do not know if Strichartz estimates holds for such potentials outside an obstacle satisfying our assumptions. 

As mentioned above, global Strichartz estimates without potential hold outside the union of strictly convex obstacles, for which assumption \ref{As:K} does not hold, as proved in \cite{Lafontaine18P} (Schr\"o\-din\-ger equation) and \cite{Lafontaine22} (Wave equation). It is reasonable to think that the Strichartz estimates remain valid in this geometrical setting with a potential satisfying Assumption \ref{As:V},  however our proof breaks down in this case.

The rest of the article is as follows: in Section \ref{S:uniform}, we prove the resolvent estimate which is needed to prove Theorem \ref{T:L2}. Section \ref{S:wave} is devoted to the wave equation, and Section \ref{S:LS} to the Schr\"odinger equation. Appendix \ref{A:free_weighted} concerns decay estimates for the free wave equations. Appendix \ref{A:zero} recalls results about the decay of energy states. Finally, in Appendix  \ref{A:ChristKiselev}, we recall Christ and Kiselev Lemma and obtain, as a consequence, of this lemma several dispersive estimates for equations \eqref{Wave_intro} and \eqref{LS_intro}.

\subsection*{Acknowledgment}
This work was initiated during Jianwei Yang's visit in the LAGA (Universit\'e Sorbonne Paris Nord) in 2018-2019, as an international chair of the Labex MME-DII.

The authors would like to thank the anonymous referee of a previous version of this article for his suggestions.

\section{The uniform limiting absorption principle for $L_V$}
\label{S:uniform}
The proof of Theorem \ref{T:L2} is  based on a uniform  estimate on the resolvent $R_V(\lambda):=(L_V-\lambda)^{-1}$. This is the main result of this section. The proof of Theorem \ref{T:L2} is completed in Subsection \ref{sub:L2wave} for the wave equation and in Subsection \ref{sub:prelim} for the Schr\"odinger equation.
\begin{proposition}
	\label{pp:alpND}
	Let $N\geq 3$, $V$ and $K$ satisfy Assumptions \ref{As:K}, \ref{As:V} and \ref{As:zero}. Let  $-2a:=\sup \sigma_{\rm pp}(L_V)<0$. Then, for all $\nu_1,\nu_2>1/2$ with 
	$\nu_1+\nu_2>2$, we have
	\begin{equation}
	\label{eq:alpND}
	\sup_{\substack{\re \lambda>-a\\ \im \lambda\neq 0}}
	\sqrt{|\lambda|+1}
	\;	\|\langle x\rangle^{-\nu_1}
	(L_V-\lambda)^{-1}\langle x\rangle^{-\nu_2}\|_{L^2(\Omega)\to L^2(\Omega)}<+\infty,
	\end{equation}
	where $\langle x\rangle=(1+|x|^2)^\frac{1}{2}$.
\end{proposition}
We split the long proof of \eqref{eq:alpND} into three subsections for convenience of reading.
\subsection{The uniform limiting absorption principle when $V\equiv 0$}
The proof of Proposition \ref{pp:alpND} relies on the analogous result for the Laplacian without potential, which we prove in this subsection:
\begin{proposition}
	\label{pp:VzeroND}
	Let $N\ge 3$ and assume that $K$ satisfies Assumption \ref{As:K}. Then for $s_1,s_2>1/2$ with $s_1+s_2>2$, we have
	\begin{equation}
	\label{eq:V=0 ND}
	\sup_{\im \lambda \neq 0}
	\sqrt{|\lambda|+1}\;\|\langle x\rangle^{-s_1}(-\Delta_D-\lambda)^{-1}\langle x\rangle^{-s_2}\|_{L^2(\Omega)\to L^2(\Omega)}<+\infty.
	\end{equation}
\end{proposition}
Before going to the proof of this proposition, let us recall a classical result for the free resolvent estimates on the whole space $\mathbb{R}^N$.
\begin{proposition}
	\label{P:BB}
	Let $s>1/2$,  $N\ge 3$ and denote $-\Delta_{\mathbb{R}^N}$ the flat Laplacian on  $\mathbb{R}^N$. Then, we have
\begin{equation}
\label{eq:freeResolvent1}
	\sup_{\im \lambda \neq 0}
\sqrt{|\lambda|}\; \left\|\langle x\rangle^{-s}(-\Delta_{\mathbb{R}^N}-\lambda)^{-1}\langle x\rangle ^{-s}\right\|_{L^2(\mathbb{R}^N)\to L^2(\mathbb{R}^N)}<+\infty.
\end{equation}
Furthermore, if $s_1,s_2>1/2$ and $s_1+s_2>2$,
\begin{equation}
\label{eq:freeResolvent}
	\sup_{\im \lambda \neq 0}
\sqrt{1+|\lambda|}\; \left\|\langle x\rangle^{-s_1}(-\Delta_{\mathbb{R}^N}-\lambda)^{-1}\langle x\rangle ^{-s_2}\right\|_{L^2(\mathbb{R}^N)\to L^2(\mathbb{R}^N)}<+\infty.
\end{equation}
\end{proposition}
\begin{proof}
	According to \cite[(b) in Theorem 1]{BaRuVe97}, if $w$ satisfies $|||w|||<+\infty$ where
	\begin{equation}
	\label{eq:V-w-c}
	|||w|||=\sup_{\mu\ge 0}\int_\mu^\infty\Bigl\{\sup_{x:\,|x|=r } w(x)\Bigr\}\frac{r}{\sqrt{r^2-\mu^2}}dr,
	\end{equation}
	then
	$$ \sup_{\im \lambda \neq 0} \sqrt{|\lambda|}\left\|\sqrt{w} (\Delta_{\R^N}+\lambda)^{-1}\sqrt{w}\right\|_{L^2(\R^N)\to L^2(\R^N)}<\infty.$$
	Thus, it suffices to verify that $w(x)=\langle x\rangle^{-2s}$ fulfills the condition \eqref{eq:V-w-c} for any $s>1/2$, which we leave to the reader.
	
	In view of \eqref{eq:freeResolvent1}, \eqref{eq:freeResolvent} will follow from 
\begin{equation}
\label{eq:freeResolvent2}
	\sup_{\im \lambda \neq 0}
 \left\|\langle x\rangle^{-s_1}(\Delta_{\mathbb{R}^N}+\lambda)^{-1}\langle x\rangle ^{-s_2}\right\|_{L^2(\mathbb{R}^N)\to L^2(\mathbb{R}^N)}<\infty,
\end{equation} 
which is also well-known. Indeed one can see it as a consequence of the uniform Sobolev estimate (see \cite[Lemma 2.2, b)]{KenigRuizSogge87})
\begin{equation}
\label{KRS87}
	\sup_{\im \lambda \neq 0}
 \left\|(\Delta_{\mathbb{R}^N}+\lambda)^{-1}\right\|_{L^{r_2}(\mathbb{R}^N)\to L^{r_1}(\mathbb{R}^N)}<\infty,
\end{equation}
where $\frac{1}{r_2}-\frac{1}{r_1}=\frac{2}{N}$, $\frac{1}{2}-\frac{1}{r_1}>\frac{1}{2N}$, $\frac{1}{r_2}-\frac{1}{2}>\frac{1}{2N}$. We choose an exponent $r_1$ such that 
\begin{equation}
 \label{defr1}
 \max\left( \frac{N}{2}-s_1,\frac{N-3}{2} \right)<\frac{N}{r_1}<\min\left( \frac{N-1}{2},s_2+\frac{N-4}{2} \right)
\end{equation} 
(which is always possible since the right-hand side of \eqref{defr1} is larger than the left-hand side by the assumptions on $s_1$ and $s_2$), and let $r_2$ such that $\frac{1}{r_2}=\frac{1}{r_1}+\frac{2}{N}$. Using H\"older inequality, we obtain, with the above conditions 
$$\|u\|_{L^{r_2}}\lesssim \left\|\langle x \rangle^{s_2}u\right\|_{L^2}, \quad \left\|\langle x\rangle^{-s_1}u\right\|_{L^2}\lesssim \|u\|_{L^{r_1}},$$ 
which, together with \eqref{KRS87}, yields
\eqref{eq:freeResolvent2}.
The proof is complete.
\end{proof}
As a consequence, we have
\begin{corollary}
	\label{crl:free_resolvent}
	Let $N\ge 3$ and $s_1,s_2>1/2$ with $s_1+s_2>2$. Then 
	$$
	\sup_{\im \lambda\neq 0} \left\| \langle x\rangle ^{-s_1} \left( \lambda+\Delta_{\R^N}\right)^{-1}\langle x\rangle^{-s_2} \right\|_{L^2(\R^N)\to H^1(\R^N)}<\infty.$$
\end{corollary}
\begin{proof}
	 This bound follows easily, using the equation satisfied by $(\lambda+\Delta_{\R^N})^{-1}\langle x\rangle^{-s_2}f$ and an integration by parts.
\end{proof}
Based on Proposition \ref{P:BB}, we are ready for  the proof of Proposition \ref{pp:VzeroND}.
\begin{proof}[Proof of Proposition \ref{pp:VzeroND}]
	We recall that we have assumed the following resolvent estimate:
	\begin{equation}
	\label{eq:cutoff}
	\sup_{\im \lambda\neq 0}
	\sqrt{1+\lvert \lambda\rvert}\,
	\bigl\lVert \chi(\Delta_D+\lambda)^{-1}\chi\bigr\rVert_{L^2(\Omega)\to L^2(\Omega)}<\infty,
	\end{equation} 
	where $\Delta_D$ is the Dirichlet Laplace operator on $\Omega$ and $\chi\in C_0^{\infty}(\R^N)$. 
	Note that \eqref{eq:cutoff} implies
	\begin{equation}
	\label{eq:cutoff2}
	\sup_{\im \lambda\neq 0}
	\bigl\lVert \chi(\Delta_D+\lambda)^{-1}\chi\bigr\rVert_{L^2(\Omega)\to H^1_{0}(\Omega)}<\infty,
	\end{equation} 
	We first prove that for any $s>1/2$, we can replace one of the cut-off functions $\chi$ by $\langle x\rangle^{-s}$, i.e.
	\begin{equation}
	\label{eq:half_cutoff}
	\sup_{\im \lambda\neq 0}
	\sqrt{1+\lvert \lambda\rvert}\,
	\bigl\lVert \langle x\rangle^{-s}(\Delta_D+\lambda)^{-1}\chi\bigr\rVert_{L^2(\Omega)\to L^2(\Omega)}<\infty,\quad s>\frac 12.
	\end{equation} 
	Indeed, let $f\in L^2$, $u=(\Delta_D+\lambda)^{-1}\chi f$. Let $\psi\in C_0^{\infty}(\R^N)$ such that $\psi=0$ close to the obstacle $K$ and $\psi(x)=1$ if $|x|$ is large. Then 
	\begin{equation*}
	(\Delta+\lambda)(\psi u)=\psi\chi f+(\Delta \psi) u +2\nabla \psi\cdot\nabla u. 
	\end{equation*} 
	Thus, using that $\psi=0$ close to $K$,
	\begin{equation}
	 \label{eq:psiu}
\psi u=\left( \Delta_{\R^N}+\lambda\right)^{-1}g\text{ where }g=\psi\chi f+(\Delta \psi) u+2\nabla \psi\cdot\nabla u.
	 \end{equation} 
	Let $\tilde{\chi}\in C_0^{\infty}(\R^N)$ such that $\tilde{\chi}=1$ on the support of $1-\psi$. By \eqref{eq:cutoff2}, $\|\tilde{\chi}u\|_{H^1_0(\Omega)}\lesssim \|f\|_{L^2(\Omega)}$. This implies $\|\langle x\rangle^{2}g\|_{L^2}\lesssim \|f\|_{L^2(\Omega)}$. Thus by Proposition \ref{P:BB} and \eqref{eq:psiu},
	$$\sqrt{1+|\lambda|}\left\|\psi \langle x\rangle^{-s} u\right\|_{L^2}\lesssim \|f\|_{L^2}.$$
	This implies \eqref{eq:half_cutoff}. 
	By duality,
	\begin{equation}
	\label{eq:half_cutoff2}
	\sup_{\im \lambda\neq 0}
	\sqrt{1+\lvert \lambda\rvert}\,
	\bigl\lVert \chi(\Delta_D+\lambda)^{-1}\langle x\rangle^{-s}\bigr\rVert_{L^2(\Omega)\to L^2(\Omega)}<\infty, \;s>1/2.
	\end{equation} 
	By the same argument, using the improved bound \eqref{eq:half_cutoff2} instead of \eqref{eq:cutoff}, we obtain the conclusion of Proposition \ref{pp:VzeroND}.
\end{proof}

\subsection{Proof of Proposition \ref{pp:alpND}}
In this subsection, we first show (see \S \ref{sub:large}) the uniform estimate \eqref{eq:alpND} for sufficiently large $\lambda$ by perturbation from the case $V=0$.

We then prove \eqref{eq:alpND} for bounded $\lambda$,  by contradiction, extracting in \S \ref{sub:elliptic} a solution $v\not\equiv0$ of the equation $L_V v=\lambda v$ belonging to a weighted $L^2$ space. We will obtain a contradiction in \S \ref{sub:end}, by showing that $v$ must be identically zero.  For this we consider,  separately the cases $\lambda \notin [0,\infty)$ (which is straightforward, since $\lambda$ is in the resolvent set of $L_V$), $\lambda \in (0,\infty)$ (where we use a refinement of the proof of the absence of embedded eigenvalue for Schr\"odinger operators going back to Agmon \cite{Agmon75}) and $\lambda=0$ (where we use the assumption that zero is not a resonance for $L_V$).

\subsubsection{Large $\lambda$}
\label{sub:large}
We first fix $\nu_1,\nu_2>1/2$ with $\nu_1+\nu_2>2$, and prove that there exists $\Lambda>0$ such that 
\begin{equation}
\label{eq:large_lambda}
\sup_{\substack{|\lambda| \ge\Lambda\\ \im \lambda \neq 0}}
\sqrt{1+\lvert \lambda\rvert}\,
\bigl\lVert \langle x\rangle^{-\nu_1}(L_V-\lambda)^{-1}\langle x\rangle^{-\nu_2}\bigr\rVert_{L^2(\Omega)\to L^2(\Omega)}<\infty.
\end{equation}
Let $f\in L^2(\Omega)$, $\lambda\in \C$ with $\im \lambda\neq 0$. Let 
$$ u=\left( L_V-\lambda \right)^{-1}\langle x\rangle^{-\nu_2}f.$$
Then
$$ (-\Delta_D+V-\lambda)u= \langle x\rangle^{-\nu_2}f$$
and $u\in D(L_V)=H^2(\Omega)\cap H^1_0(\Omega)=D(-\Delta_D)$. Thus
$$ u=(-\Delta_D-\lambda)^{-1}\left( \langle x\rangle^{-\nu_2}f-V u  \right).$$
By Proposition \ref{pp:VzeroND}, 
\begin{equation}
\label{eq:interm_res}
\left\| \langle x\rangle^{-\nu_1} u\right\|_{L^2}\lesssim \frac{1}{\sqrt{|\lambda|}} \big( \left\|\langle x\rangle^{\nu_2} Vu\right\|_{L^2}+\|f\|_{L^2} \big). 
\end{equation} 
Since $\alpha>2$, we can assume without loss of generality, that $2<\nu_1+\nu_2<\alpha$. By the decay assumption on $V$, 
$$\left\| \langle x\rangle^{\nu_2} V u\right\|_{L^2}\lesssim \left\| \langle x\rangle^{\nu_2-\alpha}u\right\|_{L^2}\lesssim \left\| \langle x\rangle^{-\nu_1}u\right\|_{L^2}.$$
Going back to \eqref{eq:interm_res} and taking $\lambda\geq \Lambda$ large enough, we obtain the desired estimate
$$\left\| \langle x\rangle^{-\nu_1} u\right\|_{L^2}\lesssim \frac{1}{\sqrt{|\lambda|}} \|f\|_{L^2}.$$

\subsubsection{Bounded $\lambda$: extraction of the solution of an elliptic equation}
\label{sub:elliptic}
In this part, we continue with the proof that for all  $\nu_1,\nu_2>1/2$ with $\nu_1+\nu_2>2$, 
\begin{equation}
\label{eq:bounded_lambda'}
\sup_{\substack{-a\leqslant \re \lambda \leqslant \Lambda\\ \im \lambda \neq 0} }
\sqrt{1+\lvert \lambda\rvert}\,
\bigl\lVert \langle x\rangle^{-\nu_1}(L_V-\lambda)^{-1}\langle x\rangle^{-\nu_2}\bigr\rVert_{L^2(\Omega)\to L^2(\Omega)}<\infty.
\end{equation} 
Taking the complex conjugate, it suffices to consider $\lambda\in\C$ such that $\im \lambda  $ is  positive. 
Moreover, by the self-adjointness of $L_V$, $\|(L_V-\lambda)^{-1}\|_{L^2\to L^2}\lesssim 1/|\im \lambda|$. 
Thus it is sufficient to prove
\begin{equation}
\label{eq:bounded_lambda}
\sup_{\substack{-a\leqslant \re \lambda \leqslant \Lambda \\ 0<\im \lambda \leq 1}  }
\sqrt{1+\lvert \lambda\rvert}\,
\bigl\lVert \langle x\rangle^{-\nu_1}(L_V-\lambda)^{-1}\langle x\rangle^{-\nu_2}\bigr\rVert_{L^2(\Omega)\to L^2(\Omega)}<\infty.
\end{equation} 
\noindent\emph{Step 1. Compactness and the Fredholm alternative.} 	
We will use the identity
\begin{equation*}
\langle x\rangle^{-\nu_1}(L_V-\lambda)^{-1}\langle x \rangle^{-\nu_2}\,T_V(\lambda)
=\langle x\rangle^{-\nu_1}(-\Delta_D-\lambda)^{-1}\langle x\rangle^{-\nu_2}.  
\end{equation*}
where 
$$T_V(\lambda)=1+ \langle x\rangle^{\nu_2} V\left(-\Delta_D-\lambda\right)^{-1}\langle x\rangle^{-\nu_2}.$$
Again, we assume (without loss of generality) that $2<\nu_1+\nu_2<\alpha$.
In view of Proposition \ref{pp:VzeroND}, it is sufficient to prove that for all $\lambda$ such that $0<\im \lambda\leq 1$ and $-a\leq \re \lambda\leq \Lambda$, $T_V(\lambda)$ is invertible, with continuous inverse, from $L^2(\Omega)$ to $L^2(\Omega)$ and that 
\begin{equation}
\label{eq:bound_TV}
\sup_{\substack{-a\leqslant \re \lambda\leqslant \Lambda \\ 0<\im \lambda\leq 1}} \left\|T_V(\lambda)^{-1}\right\|_{L^2\to L^2}<\infty.
\end{equation}
In this step, we prove the invertibility of $T_V(\lambda)$ at fixed $\lambda$. We will prove the estimate \eqref{eq:bound_TV} in the two last steps. 
Since $ (-\Delta_D-\lambda)^{-1}$ is continuous from $L^2(\Omega)$ to $H^{1}_0(\Omega)$ and we have assumed $\nu_2<\alpha$, we have 
$$ \lim_{|x|\to\infty} \langle x\rangle^{\nu_2}V(x)=0,$$
and thus (using the Rellich-Kondrachov theorem, see e.g. \cite[Subsection 5.7]{Evans10Bo}), we see that $f\mapsto \langle x \rangle^{\nu_2}V\,f$ is a compact operator from $H^2$ to $L^2$, which yields that
$\langle x\rangle^{\nu_2} V\left(-\Delta_D-\lambda\right)^{-1}\langle x\rangle^{-\nu_2}$ is a compact operator from $L^2(\Omega)$ to $L^2(\Omega)$. By the Fredholm alternative, $T_V(\lambda)$ is invertible if and only if $-1$ is not an eigenvalue, in $L^2(\Omega)$, of the operator $\langle x\rangle^{\nu_2} V\left(-\Delta_D-\lambda\right)^{-1}\langle x\rangle^{-\nu_2}$. 

Let $u\in L^2(\Omega)$ such that
$$ \langle x\rangle^{\nu_2} V\left(-\Delta_D-\lambda\right)^{-1}\langle x\rangle^{-\nu_2}u=-u.$$
Letting $f=(-\Delta_D-\lambda)^{-1}\langle x\rangle^{-\nu_2} u$, we have $f\in H^1_0(\Omega)\cap H^2(\Omega)$ and
$$ \langle x\rangle^{\nu_2} V f=-\langle x\rangle^{\nu_2} (-\Delta_D-\lambda)f.$$
Thus $(L_V-\lambda)f=0$, with $f\in D(L_V)$. Since $\im \lambda\neq 0$, this implies $f=0$ and $u=0$. Thus $-1$ is not an eigenvalue for $\langle x\rangle^{\nu_2} V\left(-\Delta_D-\lambda\right)^{-1}\langle x\rangle^{-\nu_2}$ and $T_V(\lambda)$ is invertible.

\medskip

\noindent\emph{Step 2. Existence of a special element $v\not\equiv 0$ } In this step and the next one, we continue to prove the uniform bound \eqref{eq:bound_TV} on the inverse of $T_V(\lambda)$.

 We argue by contradiction. If \eqref{eq:bound_TV} does not hold, there exists a sequence $(\lambda_n)_n$ with $-a\leq \re \lambda_n\leq \Lambda$, $0<\im \lambda_n\leq 1$ such that 
\begin{equation}
\label{eq:V11} 
\left\| T_V(\lambda_n)^{-1}\right\|_{L^2\to L^2} \geq n.
\end{equation} 
By \eqref{eq:V11}, there exists $u_n\in L^2(\Omega)$ such that 
\begin{equation}
\label{eq:V12}
\left\|u_n\right\|_{L^2(\Omega)}\leq \frac{2}{n},\quad \left\|T_V(\lambda_n)^{-1}u_n\right\|_{L^2(\Omega)}=1.
\end{equation} 
Let $v_n=\left( -\Delta_D-\lambda_n \right)^{-1}\langle x\rangle^{-\nu_2} T_V(\lambda_n)^{-1}u_n$. By \eqref{eq:V12} and Proposition \ref{pp:VzeroND}, we see that $(\langle x\rangle^{-\nu_1}v_n)_n$ is bounded in $L^2(\Omega,dx)$, i.e. that $(v_n)_n$ is bounded in $L^2\left(\Omega,\langle x\rangle^{-2\nu_1}dx\right)$. Extracting subsequences, we assume 
\begin{equation}
\label{eq:V20} 
v_n\xrightharpoonup[n\to\infty]{} v\text{ in }L^2\left(\Omega,\langle x\rangle^{-2\nu_1}dx\right),
\end{equation} 
where $v\in L^2\left(\Omega,\langle x\rangle^{-2\nu_1}dx\right)$. We will prove
\begin{gather}
\label{eq:V21} \lim_{n\to\infty} Vv_n=Vv \text{ strongly in }L^2\left(\Omega,\langle x\rangle^{2\nu_2}dx\right)\\
\label{eq:V21'} \lim_{n\to\infty} \nabla v_n=\nabla v\text{ strongly in }L^2_{\mathrm{\mathrm{loc}}}(\Omega)\\
\label{eq:V22}\lim_{n\to\infty} (\Delta_D+\lambda_n)^{-1} Vv=v\text{ strongly in }L^2\left(\Omega, \langle x\rangle^{-2\nu'}dx \right), \quad \forall \nu'>\nu_1\\
\label{eq:V23} \nabla v\in L^2_{\mathrm{\mathrm{loc}}}(\Omega),\quad v_{\restriction\partial \Omega}=0\\
\label{eq:V24}
v\not \equiv 0.
\end{gather}
By the definition of $v_n$, 
\begin{equation}
\label{eq:equation_vn}
(-\Delta_D-\lambda_n)v_n=\langle x\rangle^{-\nu_2} T_V(\lambda_n)^{-1}u_n. 
\end{equation} 
The sequence $(\lambda_n)_n$ is bounded. Since $(T_V(\lambda_n)^{-1}u_n)_n$ is bounded in $L^2(\Omega)$ by \eqref{eq:V12} and $(v_n)_n$ is bounded in $L^2(\Omega,\langle x\rangle^{-2\nu_1}dx)$, we deduce that $(\Delta_Dv_n)_n$ is bounded in $L^2(\Omega,\langle x\rangle^{-2\nu_1}dx)$. Using the Rellich-Kondrachov Theorem, we obtain that $(v_n)_n$ converges strongly to $v$ in $H^1_{\mathrm{\mathrm{loc}}}(\Omega)$, which yields \eqref{eq:V21'}. 

Also, we have 
$$\langle x\rangle^{\nu_2}V (v_n-v)=\langle x\rangle^{\nu_2} V \langle x\rangle^{\nu_1}  \langle x\rangle^{-\nu_1}(v_n-v).$$
Assuming $\nu_1+\nu_2<\alpha$, we deduce
$$\lim_{R\to\infty}\limsup_{n\to\infty}\left\|\langle x\rangle^{\nu_2}V (v_n-v)\indic_{\{|x|>R\}}\right\|_{L^2(\Omega)}=0.$$
Together with the strong convergence of $v_n$ in $L^2_{\loc}(\Omega)$ that we have already proved, we deduce \eqref{eq:V21}.

By the same strategy, we can prove
\begin{equation}
 \label{eq:V21''}
 \forall \nu'>\nu_1,\quad \lim_{n\to\infty} v_n=v \text{ strongly in }L^2\Big(\Omega,\langle x\rangle^{-2\nu'}\Big)
\end{equation} 
Since $v_n\in H^1_0(\Omega)$ for all $n$, we see that \eqref{eq:V21'} implies \eqref{eq:V23} .

Next, by \eqref{eq:equation_vn}, we see that 
$$u_n=T_V(\lambda_n)\Big( \langle x \rangle^{\nu_2} (-\Delta_D-\lambda_n)v_n \Big)=\langle x\rangle^{\nu_2} \left( -\Delta_D-\lambda_n \right)v_n+\langle x\rangle^{\nu_2} V v_n,$$
i.e.
\begin{equation}
\label{eq:V30}
(-\Delta_D-\lambda_n+V)v_n=\langle x\rangle^{-\nu_2} u_n.
\end{equation} 
This yields 
\begin{equation}
\label{eq:V30'}
v_n=\left( -\Delta_D-\lambda_n \right)^{-1}\left( \langle x\rangle^{-\nu_2}u_n-Vv_n \right).
\end{equation} 
Using that by \eqref{eq:V12}, $(u_n)_n$ converges to $0$ in $L^2(\Omega)$, and Proposition \ref{pp:VzeroND}, we deduce
$$ \lim_{n\to\infty} \left\| \langle x\rangle^{-\nu_1} (\Delta_D+\lambda_n)^{-1}\langle x\rangle^{-\nu_2}u_n\right\|_{L^2}=0.$$
Furthermore
$$\left( \Delta_D+\lambda_n \right)^{-1}\left( Vv-Vv_n \right)=\left( \Delta_D+\lambda_n \right)^{-1}\langle x\rangle^{-\nu_2} \langle x\rangle^{\nu_2} \left( Vv-Vv_n \right),$$
and thus, by \eqref{eq:V21} and Proposition \ref{pp:VzeroND}
$$\lim_{n\to\infty} \left\|\langle x\rangle^{-\nu_1} \left( \Delta_D+\lambda_n \right)^{-1} \left( Vv-Vv_n \right)\right\|_{L^2}=0.$$
Combining with \eqref{eq:V30'} and \eqref{eq:V21''} we deduce \eqref{eq:V22}.

Next, we observe that $T_V(\lambda_n)T_V(\lambda_n)^{-1}u_n=u_n$, and thus
$$ T_V(\lambda_n)^{-1}u_n+\langle x\rangle^{\nu_2}V \left( -\Delta_D-\lambda_n \right)^{-1}\langle x\rangle^{-\nu_2} T_V\left( \lambda_n \right)^{-1}u_n=u_n.$$
By the definition of $v_n$, 
$$T_V(\lambda_n)^{-1}u_n+\langle x\rangle^{\nu_2} Vv_n=u_n.$$
Thus by \eqref{eq:V21}, $T_V(\lambda_n)^{-1}u_n$ converges to $-\langle x\rangle^{\nu_2}Vv$ strongly in $L^2$. 

Since $\left\| T_V(\lambda_n)^{-1}u_n\right\|_{L^2}=1$, we deduce $\left\| \langle x\rangle^{\nu_2}Vv\right\|_{L^2}=1$, which implies \eqref{eq:V24}.

\subsubsection{Bounded $\lambda$: end of the proof} 
\label{sub:end}
In this part, we conclude the proof of the uniform bound  \eqref{eq:bound_TV}. Extracting subsequences, we may assume
$$\lim_{n\to\infty}\lambda_n=\lambda,\quad -a\leq  \re \lambda\leq \Lambda,\quad 0\leq \im \lambda\leq 1.$$
We distinguish between three cases according to the values of $\lambda$, .

\smallskip

\noindent\emph{Case 1: $\lambda\notin [0,\infty)$.} Thus $\lambda$ is not in the spectrum of $-\Delta_D$. Using the continuity of the resolvent of $-\Delta_D$ on its resolvent set, and \eqref{eq:V22}, we deduce
\begin{equation}
\label{eq:V40}
v=(\Delta_D+\lambda)^{-1}Vv.
\end{equation} 
Since $Vv\in L^2$, \eqref{eq:V40} implies $v\in D(-\Delta_D)=H^2(\Omega)\cap H^1_0(\Omega)$, i.e.
$$L_Vv=\lambda v,$$
a contradiction since by \eqref{eq:V24}, $v$ is not identically $0$, and  $\re \lambda \geq -a$ implies that $\lambda$ is not an eigenvalue of $L_V$. 

\smallskip

	\noindent\emph{Case 2: $\lambda\in (0,\infty)$.} 
Taking the limit (in the distributional sense) of \eqref{eq:V30}, we obtain
\begin{equation}
\label{eq:V60}
(-\Delta +V-\lambda)v=0
\end{equation} 
We shall prove that 
\begin{equation}
\label{eq:embedd}
v\in L^2(\Omega).
\end{equation}
Then,  $v\equiv 0$ by the main result of Ionescu and Jerison \cite{IoJe03}  that a function $v\in H^1_{\mathrm{loc}}(\Omega)$ satisfying \eqref{eq:embedd} and $L_V v=\lambda v$ for some $\lambda>0$ must vanish identically. This contradicts \eqref{eq:V24}. \smallskip

To prove \eqref{eq:embedd} we will use results from \cite[Chapter 14]{Hormander.T2} (which essentially comes from \cite{Agmon75}). We will need the two function spaces $B$ and $B^*$ which are defined in Section 14.1 of \cite{Hormander.T2} as the subspaces of $L^2_{loc}(\R^N)$ such that the following norms are finite:
\begin{align}
\label{defB}
\|\varphi\|_{B}&=\left( \int_{|x|\leq 1}|\varphi(x)|^2dx \right)^{1/2}+\sum_{j\geq 0}\left( 2^j \int_{2^j<|x|<2^{j+1}} |\varphi(x)|^2dx\right)^{1/2}\\
\label{defB*}
\|\varphi\|_{B^*}&=\left( \int_{|x|\leq 1}|\varphi(x)|^2dx \right)^{1/2}+\sup_{j\geq 0}\left( 2^{-j} \int_{2^j<|x|<2^{j+1}} |\varphi(x)|^2dx\right)^{1/2}.
\end{align}
One has the following continuous embeddings for any $s>1/2$:
\begin{equation}
 \label{inclusion}
 L^2\left(\R^N,\langle x\rangle^{2s}dx\right)\subset B,\quad  B^*\subset L^2\left(\R^N,\langle x\rangle^{-2s}dx\right).
\end{equation} 
From Theorem 14.3.2 of \cite{Hormander.T2}\footnote{In \cite{Hormander.T2}, the definition of a \emph{weaker} polynomial, that appears in the statement of Theorem 14.3.2, is given in Definition 10.3.4, using the notation $\tilde{P}$ defined in (10.1.7) of this book.}, $(-\Delta_{\R^N}-\lambda\pm i0)^{-1}$, $\lambda>0$, defined as the limits of the resolvent on the real axis, which are bounded operators from $L^2_{\comp}$ to $L^2_{\loc}$ extend to bounded operators from $B$ to $B^*$.

We divide the proof into three steps.

\smallskip

\noindent\emph{Step 1: outgoing condition.}

 We prove (as a consequence of the fact that $\im\lambda_n>0$ for all $n$), that $v$ satisfies an outgoing condition at infinity.  Indeed, let $\chi \in C^{\infty}(\R^N)$ such that $\chi=0$ in a neighborhood of $K$, and $\chi(x)=1$ for large $x$. Let
\begin{equation}
\label{eq:w_n}
w_n =(\Delta_D+\lambda_n)^{-1}(Vv).
\end{equation}
Taking a smaller $\nu_1$ if necessary and using that $\nu_2>1/2$ and $\alpha>2$, we can assume $\nu_1<\alpha-1/2$.
We fix $\nu'$ such that $\nu_1<\nu'<\alpha-1/2$.
By \eqref{eq:V22}, $w_n$ converges to $v$, strongly in $L^2\left( \langle x\rangle^{-2\nu'} dx \right)$. By the equation $\Delta_D w_n+\lambda_n w_n=Vv$ and \eqref{eq:V60}, we see that this convergence also holds locally in $H^2$. As a consequence,
\begin{multline*}
(\Delta +\lambda_n)(\chi w_n)=(\Delta \chi)w_n+2\nabla \chi\cdot\nabla w_n+\chi V v\\
\underset{n\to\infty}{\longrightarrow} f:=(\Delta \chi)v+2\nabla \chi\cdot\nabla v+\chi Vv\text{ in }L^2\left( \langle x\rangle^{2\nu}dx \right),\quad \nu:=\alpha-\nu'>1/2.
\end{multline*}

Using the free resolvent bound (Proposition \ref{P:BB}), we see that 
$$\chi w_n=\left( \Delta_{\R^N}+\lambda_n \right)^{-1} f+o(1)$$
in $L^2\left( \langle x\rangle^{-2\tilde{\nu}}dx \right)$, for all $\tilde{\nu}>2-\nu$ such that $\tilde{\nu}>1/2$. As a consequence  
\begin{equation}
\label{eq:outgoing}
\chi v=\left(\Delta_{\R^N}+(\lambda+i0) \right)^{-1}f.
\end{equation}

\smallskip

\noindent\emph{Step 2: incoming condition.}
We next prove that $v$ also satisfies an incoming condition: 
\begin{equation}
\label{eq:incoming}
\chi v=\left( \Delta_{\R^N}+(\lambda-i0) \right)^{-1}f.
\end{equation}
Note that $f\in B$ by \eqref{inclusion}, and that $\left(\Delta_{\R^N}+(\lambda+i0) \right)^{-1}$ is a bounded operator from $B$ to $B^*$ by Theorem 14.3.2 of \cite{Hormander.T2}. Thus, by \eqref{eq:outgoing}, $\chi v$ is an element of $B^*$. Using Corollary 14.3.9 of \cite{Hormander.T2}, we can deduce \eqref{eq:incoming} from \eqref{eq:outgoing} if $\int_{\R^N} \chi v \overline{f}$ is real. We compute the imaginary part of this quantity using the definition of $f$ and the fact that $V$ is real-valued:
\begin{equation*}
 \im\int\chi v \overline{f}=2\im \int\nabla \chi \cdot\nabla \overline{v} \chi v=\im \int \nabla \chi^2\cdot\nabla \overline{v} v=-\im\int \chi^2\Delta \overline{v} v=0
\end{equation*}
by the equation \eqref{eq:V60}.

\smallskip

\noindent\emph{Step 3: decay of $v$.}
In this last step, we conclude the proof of \eqref{eq:embedd}, using that if $v$ and $f$ satisfy both \eqref{eq:outgoing} and \eqref{eq:incoming}, $v$ decays in some sense at most as $|x|f$ at infinity (see Theorem 14.3.7 of \cite{Hormander.T2}). In this Theorem, this decay is formulated using two spaces $B$ and $B^*$  defined in \eqref{defB}, \eqref{defB*}.

Let $\eps>0$ (that will ultimately go to $0$). Let 
$$\mu_{\eps}(\sigma)=(1+\sigma)(1+\eps \sigma)^{-1},$$
that satisfies $0<(1+\sigma)\mu_{\eps}'(\sigma)<\mu_{\eps}(\sigma)$ for all $\sigma>0$. By Theorem 14.3.7 of \cite{Hormander.T2}, since $\chi v$ satisfies \eqref{eq:outgoing} and \eqref{eq:incoming} and $f\in L^2\left( \R^N,\langle x\rangle^{2\nu}dx \right)\subset B$, we have
$$ \big\|\mu_{\eps}(|x|) \chi v\big\|_{B^*}\leq C \big\|\mu_{\eps}(|x|) f\big\|_{B},$$
where the constant $C$ is independent of $\eps$. Using the embeddings \eqref{inclusion} with $s=\frac{3}{4}$, we obtain 
$$ \int_{\R^N} \langle x\rangle^{-3/2} \mu_{\eps}^2(|x|)|\chi v|^2dx\leq C\int_{\R^N} \langle x\rangle^{3/2} |f|^2 \mu_{\eps}^2(|x|)dx.$$
Combining with the definitions of $f$ and $\chi$ and the fact that $\mu_{\eps}(|x|)\leq 1+|x|\leq \sqrt{2}\langle x\rangle$, we deduce that for large $R$, there exists a constant $C_R$ such that
\begin{multline}
 \int_{\R^N} \langle x\rangle^{-3/2} |\chi v|^2\mu_{\eps}^2(|x|)dx\\
 \leq C_R \int_{\substack{|x|\leq R\\ x\in \Omega}}\left( |v|^2+|\nabla v|^2 \right)dx+C\int_{|x|\geq R} \langle x\rangle^{3/2} V^2 |\chi v|^2 \mu_{\eps}^2(|x|) dx\\
 \leq C_R \int_{\substack{|x|\leq R\\ x\in \Omega}}\left( |v|^2+|\nabla v|^2 \right)dx+C\sup_{|x|\geq R}\left(\langle x\rangle^{3} V^2(x)\right)  \int_{|x|\geq R} \langle x\rangle^{-3/2} |\chi v|^2 \mu_{\eps}^2(|x|) dx.
\end{multline}
By the decay assumption \eqref{decayV} on $V$, we can choose $R$ large enough so that $C\sup_{|x|\geq R}\left(\langle x\rangle^{3} V^2(x)\right)\leq \frac{1}{2}$. We thus obtain, uniformly with respect to $\eps>0$:
$$ \int_{\R^N} \langle x\rangle^{-3/2} |\chi v|^2\mu_{\eps}^2(|x|)dx\leq 2C_R \int_{\substack{|x|\leq R\\ x\in \Omega}}\left( |v|^2+|\nabla v|^2 \right)dx.$$
Letting $\eps\to 0$ we deduce (using also that $v\in L^2_{loc}(\Omega)$)
$$ \int_{\Omega} \langle x\rangle^{1/2} |v|^2dx<\infty,$$
which proves \eqref{eq:embedd}, concluding the proof in Case 2.

\smallskip

\noindent	\emph{Case 3: $\lambda=0$.}

By \eqref{eq:V22}, $v\in L^2\left( \langle x\rangle^{-2\nu'}dx \right)$, for all $\nu'>\nu_1$, $\Delta v=Vv$, $v_{\restriction \partial \Omega}=0$ and $v$ is not identically $0$. 
By Proposition \ref{P:decay_resonance} in the appendix, $v$ is a zero energy state, which yields a contradiction.
\medskip

\section{Dispersive estimates for the wave equation}
\label{S:wave}

In all this section we will assume that $N\geq 3$, that $K$ is nontrapping in the sense that Assumption \ref{As:K} holds, that $V$ satisfies Assumption \ref{As:V} and that $L_V$ has no zero energy state (Assumption \ref{As:zero}).
If $v$ is a function of time and space, we will denote $\vec{v}=(v,\partial_tv)$

\subsection{Energy estimates}


We first give an energy estimate that will be needed in the proof. We denote by $\dot{H}^1_V=P_c(\dot{H}^1_0)$, with the scalar product and norm:
\begin{equation*}
\left(\varphi,\psi\right)_{\dot{H}^1_V}=\int_{\Omega} \nabla \varphi\cdot \nabla \overline{\psi}+\int_{\Omega} V \varphi\overline{\psi},\quad \|\varphi\|^2_{\dot{H}^1_V}=(\varphi,\varphi)_{\dot{H}^1_V}.
\end{equation*} 
\begin{lemma}
	\label{lem:equiv_norm}
	We have 
	\begin{equation}
	\label{eq:equi-en}
	\forall u\in \dot{H}_V^1,\quad \|u\|_{\dot{H}^1_V}\approx \|u\|_{\dot{H}^1_0}.
	\end{equation} 
\end{lemma}
\begin{proof}
	Since the essential spectrum of $L_VP_c$ is $[0,\infty)$ and it has no negative eigenvalue, we have 
	\begin{equation}
	\label{seminorm}
	\forall \varphi\in \dot{H}^1_V,\quad
	\int_{\Omega} |\nabla \varphi|^2+\int_{\Omega} V |\varphi|^2\geq 0,
	\end{equation} 
	and thus $\|\cdot\|_{\dot{H}^1_V}$ is a semi-norm on $\dot{H}^1_V$.
	By the self-adjointness, and the Hardy inequality, we have
	\begin{equation}
	\label{eq:equi_en1}
	\int_{\Omega} |\nabla u|^2+\int_{\Omega} V(x)
	\,\lvert u \rvert^2dx
 \lesssim \int_{\Omega} |\nabla u|^2.
	\end{equation}
	
	We next prove
	\begin{equation}
	\label{eq:equi_en2}
	\int_{\Omega}|\nabla u|^2\lesssim\int_{\Omega} |\nabla u|^2+\int_{\Omega} V |u|^2
	,\quad \forall \;u\in P_c\dot{H}^1_0(\Omega),
	\end{equation}
	We argue by contradiction, using the non-existence of zero energy states for $L_V$. Assume that \eqref{eq:equi_en2} fails.
	Then, there exists a sequence of functions $(u_\nu)_{\nu}$ in $\dot{H}^1_V$ such that 
	\begin{equation}
	\label{eq:cv}
	\|u_\nu\|_{\dot{H}^1_0(\Omega)}=1,\quad\lim_{\nu\to\infty}     \int_{\Omega}|\nabla u_{\nu}|^2+\int_{\Omega} V
	\,\lvert u_{\nu} \rvert^2=0.
	\end{equation}
	Thus there exists a subsequence, still denoted by $(u_\nu)_{\nu}$,  converging weakly to some function $u_*$ in $\dot{H}^1_0$. By \eqref{eq:cv}, using the decay assumption \eqref{decayV} on $V$ we have 
	$$\int_{\Omega}V|u_*|^2=\lim_{\nu\to\infty} \int_{\Omega} V|u_{\nu}|^2=-1,$$
	and thus $u_*\neq 0$. Using the weak convergence of $(u_{\nu})_{\nu}$, we also see that $\int_{\Omega} u_*Y_j=0$ for all $j\in \{1,\ldots,J\}$ and thus $u_*\in P_c\dot{H}^1_0(\Omega)$. Furthermore, by a standard property of the weak convergence, 
	$$ \int_{\Omega}|\nabla u_*|^2\leq \liminf_{\nu\to\infty} \int_{\Omega}|\nabla u_{\nu}|^2$$
	and thus (using also \eqref{seminorm})
	$$\int_{\Omega}|\nabla u_*|^2+\int_{\Omega}V|u_*|^2=0.$$
	By \eqref{seminorm} and Cauchy-Schwarz's inequality, we obtain 
	\begin{equation}
	\label{eq:OP}
	 \forall \varphi\in C_0^{\infty}(\Omega),\quad   \int_{\Omega}\nabla u_*\cdot\nabla P_c\varphi+\int_{\Omega}Vu_*P_c\varphi=0.
	\end{equation}
	This can be rewritten as
	$$ \forall \varphi\in C_0^{\infty}(\Omega),\quad   \left\langle u_*,L_V\varphi\right\rangle_{L^2}=\left\langle u_*,L_VP_c\varphi\right\rangle_{L^2}=0,$$
	where to obtain the left-hand side inequality, we have used that $P_cu_*=u_*$ and that $P_c$ is self-adjoint for the $L^2$ scalar product. Thus $L_Vu_*=0$ in the sense of distributions. Since $u_*\in \dot{H}^1_0$, we deduce from Hardy's inequality that $u_*$ is a zero energy state for $L_V$, which contradicts the assumption that zero is not an energy state for $L_V$.
	The proof is complete.
\end{proof}

\begin{corollary}
	\label{cor:energy-G}
	There exists a constant $C$ such that for all solutions $u$ of \eqref{Wave_intro}, \eqref{Wave_ID}, we have
	\begin{equation}
	\label{eq:energy-est}
	\lVert\PPP_c\vec{u}\rVert_{L^\infty(\R;  \dot{ H}^1_0(\Omega)\times L^2(\Omega))}\le C\bigl( \|\PPP_c(u_0,u_1)\|_{ \dot{ H}^1_0(\Omega)\times L^2(\Omega)}
	+\| P_c f\|_{L^1_tL^2_x(\R\times \Omega)}\bigr).
	\end{equation}
\end{corollary}
\begin{proof}
	We denote by $\HHH_V=\PPP_c(\HHH)=\dot{H}^1_V\times P_cL^2$, with the scalar product:
	\begin{equation*}
	\left((u_0,u_1),(v_0,v_1)\right)_{\HHH_V}=\int _{\Omega}u_1\overline{v_1}+\int_{\Omega} \nabla u_0\cdot \nabla \overline{v}_0+\int_{\Omega} V u_0\overline{v}_0=(u_0,v_0)_{\dot{H}^1_V}+(u_1,v_1)_{L^2}.
	\end{equation*} 
	By Lemma \ref{lem:equiv_norm},
	\begin{equation}
	\label{eq:2}
	\forall (v_0,v_1)\in \HHH_V,\quad \|(v_0,v_1)\|_{\HHH_V}\approx \|(v_0,v_1)\|_{\HHH}.
	\end{equation} 
	
	Since  $P_c$ commutes with $L_V$, we can assume $\PPP_c(u_0,u_1)=(u_0,u_1)$, and for all $t\in \R$, $P_cf(t)=f(t)$. As as consequence, we also have $\PPP_c\vec{u}(t)=\vec{u}(t)$ for all $t$.
	
	In view of the self-adjointness of $L_V$, direct computation yields
	\begin{equation*}
	\frac{1}{2}\frac{d}{dt}\left\|(u(t),\partial_tu(t))\right\|^2_{\HHH_V}
	=\langle \partial_{t} u(t), f(t)\rangle_{L^2}
	\end{equation*}
	and integrating in time, we obtain
	\begin{equation*}
	\left\|\vec{u}(t)\right\|^2_{\HHH_V}\leq
	\left\|\vec{u}(0)\right\|^2_{\HHH_V}
	+2\int_0^t\langle \partial_{s} u(s), f(s)\rangle_{L^2}ds.
	\end{equation*}
	Letting $\displaystyle\phi(t)=\max_{0\leqslant s\leqslant t}\|\vec{u}(s)\|_{\HHH_V}$ and using  Cauchy-Schwarz's inequality, we have
	\begin{equation}
	\label{eq:energyest}
	\phi(t)^2\le \phi(0)^2+2\phi(t)\int_{0}^{t}\lVert f(s)\rVert_{L^2(\Omega)}ds,
	\end{equation}
	which yields \eqref{eq:energy-est}.
	The proof is complete.
\end{proof}
We define $\dot{H}^{-1}=\dot{H}^{-1}(\Omega)$ as the dual space of $\dot{H}^{1}_0$ with respect to the pivot space $L^2(\Omega)$.
\begin{corollary}
 \label{cor:H-1}
For all $f\in \dot{H}^{-1}$ with $P_cf=f$, there exists a unique $u\in \dot{H}^1_V$ such that $L_V u=f$ in the sense of distribution.
 \end{corollary}
\begin{notation}
 If $f$ and $u$ are in the corollary, we will denote $u=L_V^{-1}f$.
\end{notation}
\begin{proof}
 By Riesz representation theorem and Lemma \ref{lem:equiv_norm}, for all $f\in \dot{H}^{-1}$, there exists a unique $u\in \dot{H}^1_V$ such that 
 $$ \forall \varphi\in \dot{H}^1_V,\quad (u,\varphi)_{\dot{H}^1_V}=\langle f,\varphi\rangle_{\dot{H}^{-1},\dot{H}^1_0}.$$
 If furthermore $P_cf= f$, we obtain
 $$\forall \varphi\in C^{\infty}_0(\Omega), \quad \int_{\Omega} \nabla u\cdot\nabla \varphi+\int_{\Omega}V u\varphi=\int_{\Omega} f\varphi,$$
 which concludes the proof.
\end{proof}
\begin{corollary}
\label{cor:H-1_wave}
For all $(u_0,u_1)\in H^1_0\times (L^2\cap \dot{H}^{-1})$ such that $\PPP_c(u_0,u_1)=(u_0,u_1)$, the solution $u$ of \eqref{Wave_intro} with $f=0$ and initial data $(u_0,u_1)$ satisfies
$$\sup_{t\in \R} \|\vec{u}(t)\|_{L^2\times \dot{H}^{-1}}\lesssim \|(u_0,u_1)\|_{L^2\times\dot{H}^{-1}}$$
\end{corollary}
\begin{proof}
 Consider the solution $v$ of \eqref{Wave_intro} with right-hand side $f=0$, and initial data $(v_0,v_1)$ defined by 
 $$ v_1=u_0,\quad v_0=L_V^{-1}u_1.$$
 By uniqueness in the Cauchy problem in $\dot{H}^1_0\times L^2$, we see that $\partial_tv=u$, $L_Vv=\partial_tu$, and the conclusion of the corollary follows from Corollary \ref{cor:energy-G}.  
\end{proof}
By Corollary \ref{cor:H-1_wave}, we can extend the wave equation \eqref{Wave_intro} to initial data $(u_0,u_1)\in L^2\times \dot{H}^{-1}$, at least when $f$ is identically $0$.
\subsection{$L^2$-integrability of the weighted energy}

\label{sub:L2wave}

In this section, we prove Theorem \ref{T:L2} for the wave equation, i.e. the inequality \eqref{eq:lse}.


	The argument is inspired by the one in \cite{Burq03}, where the inequality with $V=0$ and a compactly supported weight is proved. However we found a small error in the proof of \cite{Burq03} in the homogeneous case (i.e. when $f=0$ and $(u_0,u_1)$ is not zero). Using the notations of this article, p.1678 the norm $\|\mathcal{A}^s \cdot\|_{L^2}$ is not adapted to the space $\dot{\mathcal{H}}^s$. This also affects the formula giving the adjoint in (2.6) and thus the whole $T T^*$-argument. 
	The $TT^*$ argument can be fixed \cite{BurqPC}. However  we were not able do it in a simple way, and propose a different elementary proof for this case (see Step 3 below) that uses finite speed of propagation and the corresponding inequality for the free wave equation (without obstacle). We note that this proof also works in the case $V=0$, and thus provides a complete proof of Theorem 3 of \cite{Burq03} in the case $\gamma=0$. An alternative approach would be to use the generalization of the theory of Kato's smoothing to wave equation developed in \cite{DAncona14}.
	
	The proof is divided into three steps. In all the proof, we fix $\nu_1,\nu_2>1/2$ such that $\nu_1+\nu_2>2$.
	\medskip

	\noindent\emph{ Step 1. Resolvent estimates.}
	We start by proving that there exist $\eps_0>0$, $C>0$ such that for $0<|\im \lambda|\leq \eps_0/(1+|\re\lambda|)$. Again, in the sequel of the proof, I systematically changed $\nu$ into $\nu_1$ or $\nu_2$ without indicating the changes.
	\begin{align}
	\label{TD21}
	\| \langle x\rangle^{-\nu_1} (L_V-\lambda^2)^{-1}\langle x\rangle^{-\nu_2}\|_{L^2(\Omega)\to L^2(\Omega)}& \leq C/(1+|\lambda|)\\
	\label{TD22}
	\|\langle x\rangle^{-\nu_1}(L_V-\lambda^2)^{-1}\langle x\rangle^{-\nu_2}\|_{L^2(\Omega)\to \dot{ H}^1_0(\Omega)}&\leq C.
	\end{align}
	Note that the condition on $\lambda$ implies $\re (\lambda^2)\geq -\eps_0^2$, $\im (\lambda^2)\leq 2\eps_0$. Thus \eqref{TD21} follows immediately from Proposition \ref{pp:alpND} taking $\eps_0>0$ small enough.
	
	We next prove \eqref{TD22}.
	To this end, we let $F\in L^2(\Omega)$ and $U=(L_V-\lambda^2)^{-1}\langle x\rangle^{-\nu_2} F$, so that 
	\begin{equation}
	 \label{TD21'}
	 \left\|\langle x\rangle^{-\nu_1} U\right\|_{L^2(\Omega)}\lesssim (1+|\lambda|)^{-1}\|F\|_{L^2(\Omega)}
	\end{equation} 
	by \eqref{TD21}.
	Then, integrating by parts, we have
	\begin{align*}
	\int_{\Omega}\langle x\rangle^{-\nu_2}F(x)\langle x\rangle^{-2\nu_1}\overline{U(x)}dx
	=&\int_{\Omega}-\Delta U\;\overline{U}\,\langle x\rangle^{-2\nu_1}dx+\int_{\Omega}(V-\lambda^2)\langle x\rangle^{-2\nu_1}\lvert U\rvert^2dx\\
	=&\int_{\Omega}\lvert\nabla U\rvert^2\langle x\rangle^{-2\nu_1}dx-2\nu_1\int_{\Omega}\langle x\rangle^{-2\nu_1-1}\frac{x}{\lvert x\rvert}\cdot \nabla U\,\overline{U}\,dx\\
	&+\int_{\Omega}(V-\lambda^2)\langle x\rangle^{-2\nu_1}\lvert U\rvert^2dx.
	\end{align*}
	By using Cauchy-Schwarz's inequality, we obtain
	\begin{align*}
	\int \lvert\nabla U\rvert^2\langle x\rangle^{-2\nu_1}dx\le \lVert\langle x\rangle^{-\nu_2}F \rVert_{L^2}\lVert\langle x\rangle^{-2\nu_1}U \rVert_{L^2}&+\bigl(\|V\|_{L^\infty(\Omega)}+|\lambda|^2\bigr)\lVert\langle x\rangle^{-\nu_1} U \rVert_{L^2}^2\\
	&+2\nu_1\lVert \langle x\rangle^{-\nu_1}\nabla U\rVert_{L^2}\lVert \langle x\rangle^{-\nu_1} U\rVert_{L^2}.
	\end{align*}
	In view of \eqref{TD21'} and the elementary estimate $2\alpha\beta\le\alpha^2+\beta^2$, we have 
	\begin{equation}
	\lVert \langle x\rangle^{-\nu_1}  U\rVert_{\dot{H}^1(\Omega)}+
	\lVert \langle x\rangle^{-\nu_1} \nabla U\rVert_{L^2(\Omega)}\lesssim \lVert F\rVert_{L^2(\Omega)}.
	\end{equation}
	Therefore, we obtain \eqref{TD22}.

	\smallskip
	
	\noindent\emph{Step 2. The inhomogeneous equation.} We next prove \eqref{eq:lse} when $(u_0,u_1)=(0,0)$. By density, we can assume that $f$ has compact support. We denote by $v(t)=P_cu(t)$.
	Using Corollary \ref{cor:energy-G}, we deduce that $\|\vec{v}(t)\|_{\HHH}$ is uniformly bounded. 
	Writing $f=\indic_{\{t\leq 0\}} f+\indic_{\{t>0\}}f$, treating separately the two contributions and using time symmetry, we can assume that $t\geq 0$ on the support of $f$, and thus (since $(u_0,u_1)=(0,0)$), on the support of $v$.
	
	Let $v_{\eps}(t)=e^{-\eps t} v(t)$, $f_{\eps}(t)=e^{-\eps t} f(t)$, which both belong to $(L^1\cap L^2)(\R,L^2(\Omega))$. The equation \eqref{Wave_intro} writes
	$$ \partial_t^2v_{\eps}+2\eps\partial_tv_{\eps} +\eps^2v_{\eps}+L_Vv_{\eps}=P_c(f_{\eps}).$$
	Denoting by $\widehat{\phi}$ the Fourier transform in the time variable of a $L^2$-valued tempered distribution $\phi$, we obtain
	$$ \left(-\tau^2+2 i\eps\tau+\eps^2+L_V\right)\widehat{v}_{\eps}(\tau)=P_c\left( \widehat{f}_{\eps}(\tau) \right),$$
	and thus 
	\begin{equation}
	\label{TD50}
	\widehat{v}_{\eps}(\tau)=\left(-\tau^2+2 i\eps\tau+\eps^2+L_V\right)^{-1} P_c\left( \widehat{f}_{\eps}(\tau) \right).
	\end{equation} 
	We have 
	\begin{multline*}
	\la x\ra^{-\nu_1}\left( -\tau^2+2 i\eps\tau+\eps^2+L_V\right)^{-1} P_c \la x\ra^{-\nu_2}\\
	=\la x\ra^{-\nu_1}\left(  -\tau^2+2 i\eps\tau+\eps^2+L_V  \right)^{-1} \la x\ra^{-\nu_2}\left( 1+\la x\ra^{\nu_2}\left( P_c-1 \right)\la x\ra^{-\nu_2} \right),
	\end{multline*}
	Note that $\tau^2-\eps^2-2i\eps \tau=(\tau-i\eps)^2$.
	Since $(P_c-1)$ is the orthogonal projection on a finite dimensional space of smooth, exponentially decaying functions, we deduce from Step 1 treating $\tau>0$ and $\tau <0$ separately that 
	$$ |\eps|\leq \frac{c}{1+|\tau|}\Longrightarrow \left\|\la x\ra^{-\nu_1}\left(-\tau^2+2 i\eps\tau+\eps^2+L_V  \right)^{-1} \PPP_c \la x\ra^{-\nu_2}\right\|_{L^2\to H^1_0} \leq C$$
	and
	$$ |\eps|\leq \frac{c}{1+|\tau|}\Longrightarrow |\tau|\left\|\la x\ra^{-\nu_1}\left(-\tau^2+2 i\eps\tau+\eps^2+L_V  \right)^{-1} \PPP_c \la x\ra^{-\nu_2}\right\|_{L^2\to L^2} \leq C,$$
	for a small constant $c>0$.
	By \eqref{TD50}, the Plancherel theorem, and the fact that $e^{-\eps t}\leq 1$ on the support of $f$, we obtain that for all $M$, if $\eps\leq c/(1+M)$, 
	\begin{multline*}
	\left\| \la x\ra^{-\nu_1} \widehat{v}_{\eps} \right\|_{L^2\left( (-M,M),H^1_0\right)}+\left\| \tau \la x\ra^{-\nu_1} \widehat{v}_{\eps} \right\|_{L^2\left( (-M,M),L^2\right)}\\
	\leq C\left\|\la x\ra^{\nu_2}\widehat{f}_{\eps}\right\|_{L^2\left((-M,M)\times \Omega \right)}\leq C\|\la x\ra^{\nu_2}f\|_{L^2(\R\times \Omega)}.		 
	\end{multline*}
	Letting $\eps\to 0$, we deduce that 
	$$\left\| \la x\ra^{-\nu_1} \widehat{v} \right\|_{L^2\left( (-M,M),H^1_0\right)}+\left\| \tau \la x\ra^{-\nu_1} \widehat{v}\right\|_{L^2\left( (-M,M),L^2\right)}\leq C\left\|\la x\ra^{\nu_2}f\right\|_{L^2(\R\times \Omega)}.$$
	Letting $M\to\infty$, and using the Plancherel theorem again, we deduce that \eqref{eq:lse} holds when $(u_0,u_1)=(0,0)$.
	
	\smallskip
	
	\noindent\emph{Step 3. The homogeneous equation.} In this step we treat the case where $f=0$, $(u_0,u_1)\in \HHH$. We let $\nu_1>1/2$ such that $\nu_1+\alpha>7/2$. We can thus choose $\nu_2>1/2$ such that $2-\nu_1<\nu_2<\alpha-3/2$.
	
	We denote by $v(t)=P_cu(t)$, $(v_0,v_1)=\PPP_c(u_0,u_1)$, and $(\underline{v}_0,\underline{v}_1)$ the extension of $(v_0,v_1)$ to $\R^N$ by $(\uv_0,\uv_1)(x)=0$ for $x\in K$.  We let $v_F$ be the solution of the free wave equation
	\begin{equation}
	\label{FW}
	\left\{
	\begin{aligned}
	(\partial_t^2-\Delta)v_F&=0,\quad (t,x)\in \R\times \R^N\\
	\vec{v}_F(0)&=(\uv_0,\uv_1).
	\end{aligned}\right.
	\end{equation} 
	We fix $\chi\in C_0^{\infty}(\R^N)$ such that $\chi=1$, close to $K$. We let $w=v-(1-\chi) v_F$. Note that $w=v$ close to $K$, so that $w$ satisfies Dirichlet boundary conditions on $\partial \Omega$. We have:
	\begin{equation*}
	\left\{
	\begin{aligned}
	(\partial_t^2+L_V)w&=(\chi-1)Vv_F-(\Delta \chi)v_F-2\nabla \chi\cdot\nabla v_F,\quad (t,x)\in \R\times \Omega\\
	\vec{w}(0)&=\chi(v_0,v_1).
	\end{aligned}\right.
	\end{equation*} 
	We write $w=a+b$, where
	\begin{equation}
	\label{inh40}
	\left\{
	\begin{aligned}
	(\partial_t^2+L_V)a&=(\chi-1)Vv_F-(\Delta \chi)v_F-2\nabla \chi\cdot\nabla v_F,\quad (t,x)\in \R\times \Omega\\
	\vec{a}(0)&=(0,0).
	\end{aligned}\right.
	\end{equation}
	and
	\begin{equation}
	\label{inh41}
	\left\{
	\begin{aligned}
	(\partial_t^2+L_V)b&=0,\quad (t,x)\in \R\times \Omega\\
	\vec{v}(0)&=\chi(v_0,v_1).
	\end{aligned}\right.
	\end{equation} 
	Since $\nu_1>1/2$ and $\nu_2-\alpha<-3/2$, we have (see the appendix  \ref{A:free_weighted})
	\begin{equation}
	\label{inh42}
	\left\| \la x\ra^{-\nu_1}\vec{v}_F\right\|_{L^2(\R,\dot{H}^1\times L^2)}+
	\left\| \la x\ra^{\nu_2-\alpha}\vec{v}_F\right\|_{L^2(\R,H^1\times L^2)}\lesssim \left\|(v_0,v_1)\right\|_{\HHH}.
	\end{equation}
	Using the explicit form of $\PPP_c$, one see that it implies
	\begin{multline}
	\label{inh43}
\left\| \la x\ra^{-\nu_1}\PPP_c(1-\chi)\vec{v}_F\right\|_{L^2(\R,\dot{H}^1\times L^2)}\\
+\left\| \la x\ra^{\nu_2-\alpha}\PPP_c(1-\chi)\vec{v}_F\right\|_{L^2(\R,H^1\times L^2)}\lesssim \left\|(v_0,v_1)\right\|_{\HHH}.
	\end{multline}
	Let $\psi\in C_0^{\infty}(\R^N)$ such that $\psi=1$ on the support of $\chi$. By the equation \eqref{inh40} and Step 2,
	\begin{multline*}
	\left\| \la x\ra^{-\nu_1} \PPP_c\vec{a}\right\|_{L^2(\R,\HHH)}\lesssim \left\|\psi\left(|\nabla v_F|+|v_F| \right)\right\|_{L^2(\R\times \Omega)}+ \left\| \la x\ra^{\nu_2} V v_F\right\|_{L^2(\R\times\Omega)}\\
	\lesssim  \left\| \langle x\rangle^{\nu_2-\alpha} v_F\right\|_{L^2(\R\times \Omega)}+\big\|\psi |\nabla v_F|\big\|_{L^2(\R\times \Omega)},
	\end{multline*}
	By \eqref{inh42},
	\begin{equation}
	\label{inh50} 
	\left\|\la x\ra^{-\nu_1}\PPP_c \vec{a}\right\|_{L^2(\R,\HHH)}\lesssim \|(v_0,v_1)\|_{\HHH}.
	\end{equation} 
	Next, we notice that by Corollary \ref{cor:energy-G},
	\begin{equation}
	\label{inh50'}
	\left\|\la x\ra^{-\nu_1}\PPP_c \vec{b}\right\|_{L^2\left((-2,2),\HHH\right)}\lesssim \|(v_0,v_1)\|_{\HHH}. 
	\end{equation} 
	Letting $\varphi\in C^{\infty}(\R)$ with $\varphi(t)=0$ if $|t|\leq 1$, and $\varphi(t)=1$ if $|t|\geq 2$, we obtain
	\begin{equation}
	\label{inh51}
	(\partial_t^2+L_V)(\varphi(t)b)=\varphi''(t)b+2\varphi'(t)b,\quad \overrightarrow{\varphi(t)b}_{\restriction t=0}=(0,0).
	\end{equation} 
	By finite speed of propagation, the support of the right hand side of \eqref{inh51} is included in the compact set $\left\{x\in \Omega,\; \exists y\in \supp \chi\text{ s.t. }|x-y|\leq 2\right\}$. Thus by Step 2,
	\begin{equation*}
	\left\|\varphi(t)\la x\ra^{-\nu_1}\PPP_c \vec{b}\right\|_{L^2\left(\R,\HHH\right)}\lesssim \|\PPP_c(b,\partial_tb)\|_{L^2((-2,2),\HHH)}\lesssim \|(v_0,v_1)\|_{\HHH},
	\end{equation*} 
	where we have used the conservation of the energy again. Combining with \eqref{inh50'}, we deduce
	\begin{equation}
	\label{inh60}
	\left\|\la x\ra^{-\nu_1}\PPP_c \vec{b}\right\|_{L^2\left(\R,\HHH\right)}\lesssim \|(v_0,v_1)\|_{\HHH}.
	\end{equation} 
	Combining \eqref{inh43}, \eqref{inh50} and \eqref{inh60}, we deduce \eqref{eq:lse} in the case $f=0$. 
\qed

With a proof similar to the one of Corollary \ref{cor:H-1_wave}, we deduce:
\begin{corollary}
\label{cor:integrability_H-1}
Let $\nu_1>1/2$ such that $\nu_1+\alpha>7/2$.
For all $(u_0,u_1)\in H^1_0\times (L^2\cap \dot{H}^{-1})$ such that $\PPP_c(u_0,u_1)=(u_0,u_1)$, the solution $u$ of \eqref{Wave_intro} with $f=0$ and initial data $(u_0,u_1)$ satisfies
$$ \|\langle x \rangle^{-\nu_1}u(t)\|_{L^2(\R\times \Omega)}\lesssim \|(u_0,u_1)\|_{L^2\times\dot{H}^{-1}}.$$
\end{corollary}

\subsection{Proof of Theorem \ref{T:wave}}
\label{sub:mainproof}
We will use a Lemma of Christ and Kiselev \cite{ChKi01}, that we recall in Appendix \ref{A:ChristKiselev}. The idea to use this result to prove Strichartz estimates goes back to \cite{SmithSogge00}, following an idea of T.~Tao. 

In all the proof, $(q,r)$ denotes any non-endpoint admissible pair satisfying Assumption \ref{As:pqwave}.
From \cite{SmithSogge00,Burq03}, we know that the Strichartz estimates are valid when $V=0$, that is, if
\begin{equation}
\label{eqv_free}
\partial_{t}^2 v-\Delta_D v=f,\quad \vec{v}(0)=(v,\partial_tv)_{\restriction t=0}\in \HHH,\end{equation} 
then 
$$\|v\|_{L^qL^r}\lesssim \|f\|_{L^1L^2}+\|\vec{v}(0)\|_{\HHH}.$$

\smallskip

\noindent\emph{Step 1. Mixed estimate without potential.} We first observe that if $\nu>1/2$, then for any solution of \eqref{eqv_free} with $\la x \ra^{\nu}f\in L^2(\R\times\Omega)$, then
\begin{equation}
\label{mixed_Delta_D}
\|v\|_{L^qL^r}\lesssim \left\|\la x \ra^{\nu} f\right\|_{L^2(\R\times \Omega)}+\|\vec{v}(0)\|_{\HHH}. 
\end{equation}
Indeed, the case $f=0$ is already known. We can thus assume $\vec{v}(0)=(0,0)$. Denote by $\Pi_0$ the projection on the first component in $\HHH$, $\Pi_0(u_0,u_1)=u_0$, and by
$\mathcal{A}_0=\begin{pmatrix} 
0 & -i \\
-i\Delta_D & 0
\end{pmatrix}$. Then,
$$v=\Pi_0\int_0^{t}e^{i(t-s)\AAA_0}\left( 0,\la x\ra^{-\nu} \la x\ra^{\nu}f(s) \right)ds.$$
Using Corollary \ref{cor:integrability_H-1} with $V=0$ and a duality argument, we obtain that the operator 
$$ F\mapsto \int_0^{\infty} e^{-is\AAA_0}\left( 0,\la x\ra^{-\nu}F(s) \right)ds$$ 
is a bounded operator from $L^2(\R\times\Omega)$ to $\HHH$. Combining with Strichartz estimates for \eqref{eqv_free}, we obtain that
$$ F\mapsto \Pi_0\int_0^{\infty} e^{i(t-s)\AAA_0}\left( 0,\la x\ra^{-\nu}F(s) \right)ds$$
is bounded from $L^2(\R\times \Omega)$ to $L^qL^r$. By Christ and Kiselev Lemma, the operator 
$$ F\mapsto \Pi_0\int_0^{t} e^{i(t-s)\AAA_0}\left( 0,\la x\ra^{-\nu}F(s) \right)ds$$
is also bounded from $L^2(\R\times \Omega)$ to $L^qL^r$, which concludes this step. 

\smallskip

\noindent\emph{Step 2: the homogeneous case.} We next consider a solution $u$ of \eqref{Wave_intro}, \eqref{Wave_ID} with $f=0$. We let $v=P_cu$. Thus
$$ (\partial_t^2-\Delta_D)v=-Vv,\quad \vec{v}(0)\in \PPP_c\HHH.$$
By \eqref{eq:lse} in Theorem \ref{T:L2}, for any $\nu_1>3/2$, $\la x\ra^{-\nu_1}v\in L^2(\R\times \Omega)$, with $\|\la x\ra^{-\nu_1}v\|_{L^2(\R\times \Omega)}\lesssim \|\vec{v}(0)\|_{\HHH}$. By the assumption \eqref{decayV} on $V$ we deduce, choosing $\nu>1/2$ close enough to $1/2$
$$ \la x\ra^{\nu} V v\in L^2(\R\times \Omega),\;
\|\la x\ra^{\nu}V v\|_{L^2(\R\times \Omega)}\lesssim \|\vec{v}(0)\|_{\HHH}.$$
By Step 1,
$$\|v\|_{L^q(\R;L^r(\Omega))} \lesssim \|\vec{v}(0)\|_{\HHH}.$$

\smallskip

\noindent\emph{Step 3. Inhomogeneous case.} We next consider the case $f\neq 0$, $(u_0,u_1)=0$. Let
\begin{equation}
\label{defAV}
\mathcal{A}_V=\begin{pmatrix} 
0 & -i \\
iL_V & 0
\end{pmatrix}.
\end{equation} 
Let $T>0$.
By Corollary \ref{cor:H-1_wave}, the operator $(u_0,u_1)\mapsto \Pi_0 e^{it\AAA_V}\PPP_c(u_0,u_1)$ is bounded from $L^2\times \dot{H}^{-1}$ to $C^0([0,T],L^2)$. By duality,
$$ f\mapsto\int_0^T \PPP_c\left(e^{-is\AAA_V} (0,f)\right)ds$$
is bounded from $L^1([0,T],L^2(\Omega))$ to $\dot{H}^{1}_0\times L^2$. By Step 2, 
$$f \mapsto \Pi_0\int_0^{T}\PPP_c\left(e^{i(t-s)\AAA_V} (0,f)\right)ds $$
is bounded from $L^1([0,T],L^2)$ to $L^q([0,T],L^r)$. By Christ and Kiselev Lemma (see Appendix \ref{A:ChristKiselev}), the operator
$$f \mapsto \Pi_0\int_0^{t}\PPP_c\left(e^{i(t-s)\AAA_V} (0,f)\right)ds $$
is bounded from $L^1([0,T],L^2)$ to $L^q([0,T],L^r)$. Since all the bounds in this argument are uniform with respect to $T$, we obtain the conclusion of Theorem \ref{T:wave} in this case also, which concludes the proof. \qed

\section{Dispersive estimates for the Schr\"odinger equation}
\label{S:LS}
In this section we prove Theorem \ref{T:LS}. Once the resolvent estimates of Proposition \ref{pp:alpND} are known, 
the proof is an adaptation of classical arguments (see e.g. Section 2 of \cite{BuGeTz04a} and \cite{RodnianskiSchlag04}). We will consider a slightly more general setting than the one of Theorem \ref{T:LS}, allowing for Strichartz inequalities with loss, or at different level of regularity than $L^2$. More precisely, we will replace Assumption \ref{As:pqLS} by:
\begin{assumption}
\label{As:pqLS'}
The parameters $(p,q,\delta)\in (2,\infty]\times [2,\infty)\times [0,1/2)$ satisfy
\begin{equation}
 \label{admissible_LS_delta}
 \frac{N}{2}-\delta\leq \frac{2}{p}+\frac{N}{q}\leq \frac{N}{2}.
\end{equation}
Furthermore,  there exist a constant  $C>0$, and a function $\chi\in C_0^{\infty}(\R^N)$ such that $\chi=1$ in a neighborhood of $K$ and for any $\varphi_0\in H_D^{\delta_1}$,
\begin{equation}
\label{local_Strichartz_LS'}
\|\chi e^{it\Delta_D}\varphi_0\|_{L^{p}([0,1],L^{q})}\leq C\|\varphi_0\|_{H_D^{\delta}}.
\end{equation}
\end{assumption}
\begin{remark}
Note that Assumption \ref{As:pqLS'} covers Assumption \ref{As:pqLS} which corresponds to the case $\delta=0$. Also, if $(p,q,\delta)$ satisfy \eqref{admissible_LS_delta}, it follows from the classical Strichartz estimate and the Sobolev inequality $\|\varphi\|_{L^q(\R^N)}\lesssim \|(1-\Delta)^{\delta/2}\varphi\|_{L^r(\R^N)}$, where $\frac{N}{q}-\delta=\frac{N}{r}$, that
\begin{equation}
\label{subcritical_Strichartz_RN}
\|e^{it\Delta}\varphi_0\|_{L^{p}(\R,L^{q}(\R^N))}\leq C\|\varphi_0\|_{H^{\delta}(\R^N)}.
\end{equation}
The case where $\frac{N}{2}-\delta<\frac{2}{p}+\frac{N}{q}$ in Assumption \ref{As:pqLS}, characterizes a loss of derivative compared to the Strichartz estimate \eqref{subcritical_Strichartz_RN} on $\R^N$. It is satisfied in particular for any smooth nontrapping obstacle, in dimension $N=3$, for any $(p,q)$ such that $\delta>\frac{1}{3p}$ and $\frac{2}{p}+\frac{3}{q}=\frac{3}{2}$. See \cite{BlSmSo08}, and Remark 1 and Proposition 2.1 in \cite{Anton08}.

Assumption \ref{As:pqLS} is also true without loss of derivative, i.e. $\frac{N}{2}-\delta=\frac{2}{p}+\frac{N}{q}$, for any smooth nontrapping obstacle in any dimension $N\geq 3$, provided that the additional condition $\frac{1}{p}+\frac{1}{q}\leq \frac{1}{2}$ holds: see \cite{BlSmSo12}.
\end{remark}
We will prove the following more general version of Theorem \ref{T:LS}:
\begin{theorem}
 	Let $\Omega$ and $V$ satisfy Assumptions \ref{As:K}, \ref{As:V}, \ref{As:zero}. Let $(p_1,q_1,\delta_1)$, $(p_2,q_2,\delta_2)$ such that Assumption \ref{As:pqLS'} holds for both triplets. Then, there is $C>0$ such that for all solutions $\varphi$ of \eqref{LS_intro}, \eqref{LS_ID} with $(1-\Delta_D)^{(\delta_1+\delta_2)/2}\psi\in  L^{p_2'}L^{q_2'}$, one has
 \label{T:LS'}
	\begin{equation}
	\label{eq:strichartzLS'}
	\|P_c \varphi\|_{L^{p_1}L^{q_1}}\le C\Bigl(\|P_c\varphi_0\|_{H_D^{\delta_1}}+\left\|(1-\Delta_D)^{(\delta_1+\delta_2)/2}P_c \psi\right\|_{L^{p_2'}L^{q_2'}}\Bigr).
	\end{equation}
\end{theorem}
In Subsection \ref{sub:prelim} we will state and sketch the proofs of estimates that will be needed in the proof: the smoothing effect for the Schr\"odinger equation with a potential, and local-in-times estimates for the Schr\"odinger equation without potential. Subsection \ref{sub:Strichartz} is devoted to the proof of Theorem \ref{T:LS'}.
 
\subsection{Preliminary estimates}
\label{sub:prelim}
We start by proving the smoothing effect for solutions of \eqref{LS_intro} projected on the continous spectrum, i.e. the estimate \eqref{smoothing} in Theorem \ref{T:L2}. The estimate is a consequence of the following estimate on the resolvent of $L_V$ projected on the continuous spectrum:
\begin{equation}
	\label{res_smoothing}
\sup_{\im \lambda\neq 0}
	\|\langle x\rangle^{-\nu_1}
	(L_V-\lambda)^{-1}P_c\langle x\rangle^{-\nu_2}\|_{H_D^{-1/2}(\Omega)\to H^{1/2}_D(\Omega)}<+\infty, 
\end{equation} 
where $\nu_1,\nu_2>1/2$, $\nu_1+\nu_2>2$.
The fact that \eqref{res_smoothing} implies \eqref{smoothing} is standard and we omit the proof (see e.g. the proof of Proposition 2.7 in \cite{BuGeTz04a}).

 To prove \eqref{res_smoothing}, we first notice that by Proposition \ref{pp:alpND} and an argument similar to the proof of \eqref{TD22},
\begin{equation}
\label{res_smooth1}
	\sup_{\substack{\re \lambda>-a\\ \im \lambda\neq 0}}
	\;	\|\langle x\rangle^{-\nu_1}
	(L_V-\lambda)^{-1}\langle x\rangle^{-\nu_2}\|_{L^2(\Omega)\to H^{1}_0(\Omega)}<+\infty. 
\end{equation} 
Next, we notice that 
\begin{equation}
 \label{res_smooth2}
	\sup_{\substack{\re \lambda\leq -a\\ \im \lambda\neq 0}}
	\;	\|\langle x\rangle^{-\nu_1}
	(L_V-\lambda)^{-1}P_c\langle x\rangle^{-\nu_2}\|_{L^2(\Omega)\to H^{1}_0(\Omega)}<+\infty.
\end{equation} 
Indeed, let $F\in L^2(\Omega)$ and $\lambda$ such that $\im \lambda\neq 0$ and $\re\lambda\leq -a$. Let $U=(L_V-\lambda)^{-1}P_c \la x\ra^{-\nu_2}F$. Then $U\in H^1_0(\Omega)\cap H^2(\Omega)$, $P_cU=U$ and
$$ (L_V-\lambda) U=P_c \la x\ra^{-\nu_2}F.$$
Taking the real part of the scalar product by $U$ in $L^{2}(\Omega)$ of the preceding equality, we obtain
$$ -\re \lambda\|U\|_{L^2}^2+\re (L_VU,U)_{L^2(\Omega)} \lesssim \|U\|_{L^2}\left\| \la x\ra^{-\nu_2} F\right\|_{L^2(\Omega)},$$
and the bound \eqref{res_smooth2} follows from Lemma \ref{lem:equiv_norm}. Combining \eqref{res_smooth1} and \eqref{res_smooth2}, we obtain 
\begin{equation*}
	\sup_{\im \lambda\neq 0}
	\;	\|\langle x\rangle^{-\nu_1}
	(L_V-\lambda)^{-1}P_c \langle x\rangle^{-\nu_2}\|_{L^2(\Omega)\to H^{1}_0(\Omega)}<+\infty. 
\end{equation*}
By duality (reversing $\nu_1$ and $\nu_2$), we obtain the same uniform bound for the operator norm from the dual $H^{-1}$ of $H^{1}_0(\Omega)$ to $L^2(\Omega)$. The bound \eqref{res_smoothing} follows by interpolation. \qed

The Schr\"odinger equation without potential on $\R^N$: 
\begin{equation}
 \label{LSRN}
 i\partial_t\varphi+\Delta \varphi=\psi, \quad \varphi_{\restriction t=0}=\varphi_0
\end{equation} 
satisfies the assumptions of Theorem \ref{T:L2}.
In this case, the projector $P_c$ is the identity, and the inequality \eqref{smoothing} the we have proved is the usual smoothing effect on $\R^N$. Together with the Strichartz estimate \eqref{subcritical_Strichartz_RN}, we obtain (using Christ and Kiselev Lemma as in the proof below) the following mixed Strichartz/smoothing effect for solutions of \eqref{LSRN}
\begin{equation}
 \label{mixteRN}
\left\| \varphi\right\|_{L^{p}(\R,L^q(\R^N))}\lesssim \|\varphi_0\|_{H^{\delta}(\R^N)}+\|\la x\ra^{\nu} \psi\|_{L^2(\R,H^{\delta-1/2}(\R^N))},
\end{equation}
where $(p,q,\delta)$ satisfies \eqref{admissible_LS_delta}, with $p>2$, and $\nu>1/2$.

We next notice, for the Schr\"odinger equation without potential on $\Omega$, that the local homogeneous Strichartz estimates imply local-in-time inhomogeneous Strichartz estimates, as well as a mixed Strichartz estimate/local smoothing inequality.
\begin{proposition}
 \label{P:local_estimates}
 Assume that $K$ satisfies Assumption \ref{As:K} and that $(p_1,q_1,\delta_1)$ and $(p_2,q_2,\delta_2)$  satisfy Assumption \ref{As:pqLS'}.
 Assume that $V=0$. 
 Let $\nu>1/2$. Let $\varphi$, $\psi$ be a solution of \eqref{LS_intro}, \eqref{LS_ID}.  Then
 \begin{align}
 \label{LS_Strichartz1}
\left\|\varphi\right\|_{L^{p_1}\left([0,1],L^{q_1}(\Omega)\right)}&\lesssim \left\|(-\Delta_D+1)^{(\delta_1+\delta_2)/2}\psi\right\|_{L^{p_2'}\left([0,1],L^{q_2'}(\Omega)\right)}+\|\varphi_0\|_{H_D^{\delta_1}}\\
\label{LS_mixte}
\left\|\varphi\right\|_{L^{p_1}\left([0,1],L^{q_1}(\Omega)\right)}&\lesssim \left\|\la x\ra^{\nu}\psi\right\|_{L^{2}\left([0,1],H_D^{\delta_1-1/2}(\Omega)\right)}+\|\varphi_0\|_{H_D^{\delta_1}},
\end{align}
where we have assumed $\chi\psi=\psi$ in \eqref{LS_Strichartz1}.
\end{proposition}
\begin{proof}
\emph{Step 1.} We prove \eqref{LS_Strichartz1} and \eqref{LS_mixte} with $\psi=0$, i.e
\begin{equation}
\label{Step1_Strich}
\left\|\varphi\right\|_{L^{p_1}\left([0,1],L^{q_1}(\Omega)\right)}\lesssim \|\varphi_0\|_{H_D^{\delta_1}}
\end{equation}
Let $\chi\in C_0^{\infty}(\R^N)$, such that $\chi=1$ in a small neighborhood of the obstacle. By Assumption \ref{As:pqLS'}, we see that it is sufficient to prove \eqref{Step1_Strich} with $\varphi$ replaced by $v=(1-\chi)\varphi$ on the left-hand side of the inequality. We have
$$i\partial_tv+\Delta v=-2\nabla \chi\cdot\nabla v-(\Delta \chi)v,$$
where this equation can be interpreted as a Schr\"odinger equation on $\R\times \R^N$ since $v$ is supported away from the obstacle. Using \eqref{smoothing}, we have
$$  \left\|(1-\Delta)^{\delta_1/2}\left(2\nabla \chi\cdot\nabla v+(\Delta \chi)v\right)\right\|_{L^2(\R,H^{\delta_1-1/2})}\lesssim \|u_0\|_{H_D^{\delta_1}}.$$
By the mixed estimate \eqref{mixteRN}, we obtain $\|v\|_{L^{p_1}(\R,L^{q_1})}\lesssim \|u_0\|_{H_D^{\delta_1}(\Omega)}$, which concludes the proof.


\emph{Step 2.} We prove \eqref{LS_mixte} and \eqref{LS_Strichartz1}  when $\varphi_0=0$.

Since the operator $\varphi_0\mapsto \la x \ra^{-\nu}e^{i\cdot\Delta_D}\varphi_0$ is bounded from $L^2(\Omega)$ to $L^2([0,1],H^{1/2}_D(\Omega))$, we see by duality that
$\psi\mapsto \int_0^1e^{-is\Delta_D}\la x\ra^{-\nu}\psi(s)ds$ extends to a bounded operator from $L^2([0,1],H^{-1/2}_D)$ to $L^2(\Omega)$. By Step 1, we obtain that
$$\psi\mapsto (1-\Delta_D)^{-\delta_1/2} \int_0^1 e^{i(t-s)\Delta_D} \la x\ra^{-\nu_1}\psi(s)ds$$
defines a bounded operator from $L^2([0,1],H_D^{-1/2})$ to $L^{p_1}([0,1],L^{q_1})$. Using Christ and Kiselev Lemma (see Lemma \ref{L:ChristKiselev}), we obtain \eqref{LS_mixte}.

The proof of \eqref{LS_Strichartz1} is similar. Indeed, by Step 1 used twice, one obtains that
$$\psi\mapsto \int_0^{1}e^{-is\Delta_D}(-\Delta_D+1)^{-(\delta_1+\delta_2)/2}\psi(s)ds$$
is bounded form $L^{p_2'}([0,1],L^{q_2'})$ to $L^{p_1}([0,1],L^{q_1})$, and the result follows again by Lemma \ref{L:ChristKiselev}.
\end{proof}

\subsection{Proof of the global Strichartz estimates}
\label{sub:Strichartz}
In all the proof, we assume Assumptions \ref{As:K}, \ref{As:V}, \ref{As:zero}, and consider two triplets $(p_1,q_1,\delta_1)$ and $(p_2,q_2,\delta_2)$ that satisfy Assumption \ref{As:pqLS'}.

\emph{Step 1. Global Strichartz estimate for the homogeneous equation away from the obstacle.} In the two first steps we prove the conclusion of Theorem \ref{T:LS} when $\psi=0$. We let $\chi$ be as in Assumption \ref{As:pqLS'}. We fix $\theta\in C_0^{\infty}(\R^N)$, with $\theta=1$ close to $K$, such that $\chi=1$ in the support of $\theta$. We let $\varphi$ be a solution of \eqref{LS_intro}, \eqref{LS_ID} with $\psi=0$, and let $v=P_c\varphi$, which is also solution of \eqref{LS_intro}, with initial data $P_c\varphi_0$. By density we can assume that $\varphi_0\in H_0^{1}(\Omega)$. By the smoothing effect \eqref{smoothing}, for any $\nu>1/2$,
\begin{equation}
\label{vH1/2}
\left\|\la x\ra^{-\nu} v\right\|_{L^2(\R,H^{1/2}_D)}\lesssim \|P_c\varphi_0\|_{L^2(\Omega)}.
\end{equation}

In this step we bound the norm $(1-\theta)v$ in $L^{p_1}L^{q_1}$. In the next step we will consider  $\theta v$. We have 
\begin{equation}
 \label{faraway}
 i\partial_t \left((1-\theta)v\right)+\Delta\left((1-\theta)v\right)=-V(1-\theta)v -2\nabla\theta\cdot\nabla v-\Delta \theta v.
\end{equation} 
Note that $(1-\theta)v$ and the right-hand side of \eqref{faraway} are supported away from the obstacle, so that we can consider \eqref{faraway} as an equation on $\R^N$. We fix $\nu$ with $1<\nu<\alpha/2$, where $\alpha$ is given by the decay assumption \eqref{decayV} on $V$.
Using the mixed estimate \eqref{mixteRN}, we obtain
\begin{multline}
 \label{mixte1}
 \left\|(1-\theta)v\right\|_{L^{p_1}L^{q_1}}\lesssim \left\|\la x\ra^{\nu} V(1-\theta)v\right\|_{L^2(\R,H^{\delta_1-1/2})}
 +\big\|\la x\ra^{\nu} |\nabla \theta||\nabla v|\big\|_{L^2(\R,H^{\delta_1-1/2})}\\+\left\|\la x\ra^{\nu}(\Delta \theta) v\right\|_{L^2(\R,H^{\delta_1-1/2})}+\|(1-\theta)P_c\varphi_0\|_{H^{\delta_1}},
\end{multline}
where $H^{a}=H^{a}(\R^N)$. Recall that $\delta_1+1/2\in [1/2,1)$ by Assumption \ref{As:pqLS'}. Using the definition of $H^{\delta_1+1/2}_D$ as an interpolation space between $L^2(\Omega)$ and $H^{1}_0(\Omega)$, and since $H^1_0(\Omega)$ can be continuously embedded in $H^1(\R^N)$ (extending the elements of $H^1_0(\Omega)$ by $0$ in $K$) , we see that
$$\big\|\la x\ra^{\nu} |\nabla \theta||\nabla v|\big\|_{L^2(\R,H^{\delta_1-1/2})}\lesssim \|\la x\ra^{-\nu}v\|_{L^2(\R,H^{\delta_1+1/2}_D)}\lesssim \|P_c\varphi_0\|_{H^{\delta_1}_D},$$
by \eqref{vH1/2}. Furthermore, using again \eqref{vH1/2},
\begin{multline*}
\left\|\la x\ra^{\nu} V(1-\theta)v\right\|_{L^2(\R\times \R^N)}+\left\|\la x\ra^{\nu}(\Delta \theta) v\right\|_{L^2(\R\times \R^N)}\lesssim \|\la x\ra^{-\nu}v\|_{L^2(\R\times \Omega)}\lesssim \|P_c\varphi_0\|_{L^2}.
\end{multline*}
Using interpolation, it is also easy to prove $\|(1-\theta)P_c\varphi_0\|_{H^{\delta_1}}\lesssim \|P_c\varphi_0\|_{H^{\delta_1}_D}$.
As a conclusion, we obtain that when $\psi=0$,
\begin{equation}
 \label{StrichartzLS_faraway}
 \left\|(1-\theta)v\right\|_{L^{p_1}L^{q_1}}\lesssim \|P_c\varphi_0\|_{H_D^{\delta_1}}.
\end{equation} 

\emph{Step 2. Global Strichartz estimate for the homogeneous equation close to the obstacle.}
In this step we consider the contribution of $\theta v$. We have 
\begin{equation}
 \label{LSclose}
 i\partial_t(\theta v)+L_V(\theta v)= -2\nabla\theta\cdot\nabla v-\Delta \theta v.
\end{equation} 
By \eqref{vH1/2}, we have 
$$ \sum_{k\in \Z} \|\theta v\|^2_{L^2\left([k,k+1],H_D^{1/2}(\Omega)\right)}\leq \|P_c\varphi_0\|^2_{L^2(\Omega)}.$$
Let $t_k\in [k,k+1]$ such that $\|\theta v(t_k)\|_{H_D^{1/2}}=\min_{k\leq t\leq k+1} \|\theta v(t)\|_{H_D^{1/2}}$ (recall that we have assume that $\varphi_0\in H^{1}_0$, so that $v\in C^0(\R,H^1_0)$ and the min is attained). Then
\begin{equation}
 \label{ell2}
 \sum_{k\in \Z} \|\theta v(t_k)\|^2_{H_D^{1/2}}\leq \left\|P_c \varphi_0\right\|^2_{L^2(\Omega)}.
\end{equation} 
By \eqref{LSclose} and the mixed local-in-time smoothing effect/Strichartz estimate \eqref{LS_mixte}, 
\begin{equation*}
\left\|\theta v\right\|_{L^{p_1}([k,k+1],L^{q_1})}^{p_1}\lesssim \left\|2\nabla\theta\cdot\nabla v+\Delta \theta v\right\|_{L^2([k,k+1],H^{\delta_1-1/2}_D)}^{p_1}+\|\theta v(t_k)\|^{p_1}_{H_D^{\delta_1}(\Omega)}.
\end{equation*}
Since $\ell^2(\Z)\subset\ell^{p_1}(\Z)$ with continuous embedding, we have, summing up over $k$, 
\begin{multline*}
\left\|\theta v\right\|_{L^{p_1}(\R,L^{q_1})}^2
\lesssim \sum_{k\in \Z}\|\theta v(t_k)\|^{2}_{H_D^{\delta_1}}+\sum_k \left\|2\nabla\theta\cdot\nabla v+\Delta \theta v\right\|_{L^2([k,k+1],H^{\delta_1-1/2}_D)}^2
\\
\lesssim \left\|\la x\ra^{-\nu}v\right\|_{L^2(\R,H^{\delta_1+1/2}_D)}^2+\|P_c\varphi_0\|_{L^2}^2\lesssim \|P_c\varphi_0\|_{H_D^{1/2}}^2
\end{multline*}
by \eqref{ell2}, \eqref{vH1/2}. Combining with \eqref{StrichartzLS_faraway}, we obtain the Strichartz estimate 
$$ \|v\|_{L^{p_1}L^{q_1}}\lesssim \|P_c\varphi_0\|_{H_D^{\delta_1}}$$
when $\psi=0$. 

\emph{Step 3. Inhomogeneous Strichartz estimate.} It remains to prove the Strichartz estimate \eqref{StrichartzLS} when $\varphi_0=0$ and $\psi$ is not identically $0$. This is again a standard application of the homogenous Strichartz inequality, the corresponding dual inequality, and  Christ and Kiselev Lemma. We omit the proof.

\appendix

\section{Decay of the weighted energy for the free wave equation}
\label{A:free_weighted}
Consider the free wave equation 
\begin{equation}
\label{FW_A}
(\partial_t^2-\Delta)u=0,\quad (t,x)\in \R\times \R^N.
\end{equation} 
We denote $\dot{H}^1\times L^2=(\dot{H}^1\times L^2)(\R^N)$. In this appendix, we give the proof of the following well-known result for the sake of completeness:
\begin{proposition}
	Let $u$ be a solution of \eqref{FW_A} with initial data $(u,\partial_tu)(0)=(u_0,u_1)\in \dot{H}^1\times L^2$. Let $\nu>\frac{1}{2}$. Then $\la x\ra^{-\nu}\vec{u}\in L^2(\R,\dot{H}^1\times L^2)$ and 
	$$ \|\la x\ra^{-\nu}\vec{u}\|_{L^2(\R,\dot{H}^1\times L^2)}\lesssim \|(u_0,u_1)\|_{\dot{H}^1\times L^2}.$$
	Furthermore, if $\nu>3/2$, 
	$$ \|\la x\ra^{-\nu}u\|_{L^2(\R\times \R^N)}\lesssim \|(u_0,u_1)\|_{\dot{H}^1\times L^2}.$$
\end{proposition}
\begin{proof}[Sketch of proof]
	By \cite[Corollary 2.3]{SmithSogge00}, we have that for all $\chi\in C_0^{\infty}(\R^N)$, 
	\begin{equation}
	\label{with_cutoff}
	\|\chi u\|_{L^2(\R,\dot{H}^1\times L^2)}\lesssim \|(u_0,u_1)\|_{\dot{H}^1\times L^2}. 
	\end{equation} 
	In particular, denoting $|\nabla_{t,x}u|^2=(\partial_tu)^2+|\nabla u|^2$, 
	$$ \Big\| |\nabla_{t,x}u| \Big\|_{L^2(\R,L^2(\{|x|\leq 1\})}\lesssim \|(u_0,u_1)\|_{\dot{H}^1\times L^2}. $$
	Rescaling, we obtain that for all $k$,
	$$ \Big\| |\nabla_{t,x}u| \Big\|_{L^2(\R,L^2(\{|x|\leq 2^k\})}^2\lesssim 2^k\|(u_0,u_1)\|_{\dot{H}^1\times L^2}^2, $$
	which yields, for any integer $k$
	$$ \Big\|\langle x\rangle^{-\nu} |\nabla_{t,x}u| \Big\|_{L^2(\R,L^2(\{2^{k-1}\leq |x|\leq 2^k\})}^2\lesssim \left(2^{k}\right)^{1-2\nu}\|(u_0,u_1)\|_{\dot{H}^1\times L^2}^2. $$
	Summing up over $k\geq 0$ and using $2\nu>1$, we obtain 
	$$ \Big\|\langle x\rangle^{-\nu} |\nabla_{t,x}u| \Big\|_{L^2(\R\times \R^N)}^2\lesssim \|(u_0,u_1)\|_{\dot{H}^1\times L^2}^2. $$
	Which, together with \eqref{with_cutoff}, and the estimate $$\|\nabla (\langle x\rangle^{-\nu}u)\|_{L^2(\R^N)}\lesssim \|\langle x\rangle^{-\nu}|\nabla u|\|_{L^2(\R^N)}+\|\nabla (\chi u)\|_{L^2(\R^N)}$$ yields the first estimate. The proof of the second estimate is similar and we omit it.
\end{proof}

\section{Decay of zero energy state}
\label{A:zero}
In this appendix, we prove:
\begin{proposition}
\label{P:decay_resonance}
Assume $N\geq 3$ and that $V\in C^0(\overline{\Omega})$ satisfies \eqref{decayV}.
Let $v\in L^2_{\mathrm{loc}}(\Omega)$ such that $\langle x\rangle^{-\nu}v\in L^2(\Omega)$ for some $\nu<\alpha-1$. Assume 
\begin{equation}
\label{eq:elliptic}
-\Delta v+V v=0 
\end{equation} 
for large $|x|$, in the sense of distribution. Then 
\begin{equation}
\label{decay_state}
\forall \rho <\frac{N}{2}-2,\quad \langle x\rangle^{\rho} v\in L^2(\Omega). 
\end{equation} 
\end{proposition}
\begin{remark}
The example given in Remark \ref{R:example} shows that the decay obtained in Proposition \ref{P:decay_resonance} is optimal, at least in the case where $K$ is a Euclidean ball or is empty. Indeed, for all $N\geq 3$, the solution $v$ of \eqref{eq:elliptic} given in this example satisfies \eqref{decay_state} and 
\begin{equation*}
\langle x\rangle^{\frac{N}{2}-2} v\notin L^2(\Omega). 
\end{equation*} 
\end{remark}
\begin{proof}
Recall a special case of Theorem $B^*$  of Stein-Weiss \cite{SW58}: 
Let $0<\gamma<N$ for $N\ge 2$ and for $\sigma >1$,  set 
$$
(T_\gamma g)(x)=\int_{\R^N} \frac{g(y)}{|x-y|^\gamma}dy,\;\;\forall\;g\in L^{2}(|x|^{2\sigma} dx).
$$
If $\max\{\sigma,\beta\}<\frac{N}{2}$ and $\sigma+\beta+\gamma=N$, then 
$\exists\; C>0$ such that 
\begin{equation}
\label{eq:SW}
\bigl \| |x|^{-\beta}T_\gamma g\bigr\|_{L^2(\R^N)}\le C\,\bigl\||x|^\sigma g\bigr\|_{L^2(\R^N)}.
\end{equation}
Let $\chi\in C^{\infty}(\R^N)$ such that $\chi(x)=0$ in a neighborhood of $K$, $\chi(x)=1$ for large $|x|$, and \eqref{eq:elliptic} is satisfied on the support of $\chi$.
We have
\begin{equation}
\label{eq:newform}
\chi v=(\Delta_{\R^N})^{-1}f=c_NT_{N-2}f, \; 
f=(\Delta \chi)v+2\nabla\chi\cdot\nabla v+\chi Vv.
\end{equation}
We fix a small $\delta>0$ and prove that for any $\rho\leq \frac{N}{2}-2-\delta$, 
\begin{equation}
\label{decay:11}\langle x\rangle^{\rho} v\in L^2(\Omega).
\end{equation} 
Note that \eqref{decay:11} holds for $\rho=-\nu$ by our assumptions. We next prove that if \eqref{decay:11} holds for some $\rho$ with $-\nu\leq \rho<\frac{N}{2}-2-\delta$, then $\langle x\rangle^{\rho'}v$ is also in $L^2(\Omega)$, where $\rho'=\min\left( \frac{N}{2}-2-\delta,\rho+\alpha-2 \right)$ and $\alpha$ is as in \eqref{decayV}. Since $\alpha>2$, \eqref{decay:11} will follow for all $\rho\leq\frac{N}{2}-2-\delta$ after a finite number of iterations. 

Let $\gamma=N-2$, $\sigma=\min\left( \rho+\alpha,\frac{N}{2}-\delta \right)$, $\beta=2-\sigma$. We have obviously $\sigma+\beta+\gamma=N$, $\sigma <\frac{N}{2}$. The fact that $\sigma>1$ follows from the assumptions $N>2$ and $\nu<\alpha-1$, and it is also easy to check that the assumptions on $\nu$, $\alpha$ and $N$, and the fact that $\delta$ is small, imply $\beta<\frac N2$. By \eqref{decay:11}, $\langle x\rangle^{\sigma}f\in L^2(\R^N)$.
Applying \eqref{eq:SW} to \eqref{eq:newform},  we obtain 
$\langle x \rangle^{-\beta}\chi v\in L^2(\R^N)$.
Noting that $-\beta=\rho'$ and using that $v\in L^2_{\mathrm{loc}}(\Omega)$, we deduce $\langle x\rangle^{\rho'}v\in L^2(\Omega)$, which concludes the proof.
\end{proof}
\section{Christ and Kiselev Lemma and applications}
\label{A:ChristKiselev}
In this appendix we recall Christ and Kiselev Lemma \cite{ChKi01}, in the version appearing in \cite{SmithSogge00}, and prove, as a consequence of this Lemma and the results obtained above, several dispersive estimates.
\subsection{Statement of the lemma}
\begin{lemma}
\label{L:ChristKiselev}
Let $X$ and $Y$ be Banach spaces. 
 Let $1\leq r<s\leq \infty$ and $-\infty\leq a<b\leq \infty$. Let $K$ a continuous function from $\R^2$ to the space of bounded linear operators from $X$ to $Y$. Let
 $$ (A f)(t)=\int_a^b K(t,\tau)f(\tau)d\tau,$$
 and assume that $A$ is a  bounded operator from $L^r((a,b),X)$ to $L^s((a,b),Y)$, with operator norm $C$. Define the operator $\bar{A}$ by
 $$\bar{A} f(t)=\int_a^t K(t,\tau)f(\tau)d\tau.$$
 Then $\bar{A}$ extends to a bounded operator from $L^r((a,b),X)$ to $L^s((a,b),Y)$, with operator norm $\leq \frac{2C\theta^2}{1-\theta}$, where $\theta=2^{\frac 1s-\frac 1r}$.
\end{lemma}
See \cite[Proof of Lemma 3.1]{SmithSogge00} for the proof.
\subsection{Mixed inequalities}
\label{AA:mixed}
We state here inequalities combining Strichartz inequalities of Theorems \ref{T:wave} and \ref{T:LS} with the $L^2$ integrability property of Theorem \ref{T:L2}. We omit the proofs, that are close to Step 1 of the proof of Theorem \ref{T:wave} in Subsection \ref{sub:mainproof}, and to the proof of Theorem \ref{T:Strichartz_H-1} below.
\begin{theorem}
With the assumptions and notations as in Theorem \ref{T:wave}, one has, for any $\nu>1/2$ such that $\nu+\alpha>7/2$,
\begin{gather}
\label{mixed_wave_1}
\|\PPP_c\vec{u}\|_{L^{\infty}_t(\R,\dot{H}^1_0\times L^{2})}+\|P_c u\|_{L^pL^q}\lesssim \lVert \PPP_c(u_0,u_1)\rVert_{\dot{ H}^1_0\times L^2}+\lVert \langle x\rangle^\nu P_c f\rVert_{L^2_{t,x}}\\
\label{mixed_wave_2}
\|\langle x\rangle^{-\nu}P_c u\|_{L^2_{t,x}}\lesssim \lVert \PPP_c(u_0,u_1)\rVert_{\dot{ H}^1_0\times L^2}+\lVert P_c f\rVert_{L^1L^2}.
\end{gather} 
\end{theorem}
\begin{theorem}
With the assumptions and notations as in Theorem \ref{T:LS}, one has, for any $\nu>1/2$,
\begin{gather}
	\label{mixed_LS_1}
	\|P_c \varphi\|_{L^{p_1}L^{q_1}}\le C\Bigl(\|P_c\varphi_0\|_{L^2}+\|\langle x\rangle^{\nu}P_c \psi\|_{L^{2}(\R,H^{-1/2}_D)}\Bigr).\\
	\label{mixed_LS_2}
\|\langle x\rangle^{-\nu}P_c \varphi\|_{L^{2}(\R,H^{1/2}_D)}\le C\Bigl(\|P_c\varphi_0\|_{L^2}+\|P_c \psi\|_{L^{p_2'}L^{q_2'}}\Bigr).\end{gather}
\end{theorem}
\subsection{Strichartz estimates for initial data with lower regularity for the wave equation}
\label{AA:lower}
In this subsection, we state Strichartz estimates for the wave equation \eqref{Wave_intro} with initial data in Sobolev spaces with lower regularity. We start with solutions with initial data in $L^2\times \dot{H}^{-1}$:
\begin{theorem}
 \label{T:Strichartz_H-1}
 Let $\Omega$ and $V$ satisfy Assumptions \ref{As:K}, \ref{As:V}, \ref{As:zero}. Let $(p,q)$ such that Assumption \ref{As:pqwave} holds. Then, there is $C>0$ such that for all solutions $u$ of \eqref{Wave_intro} with $f\in  L^{p'} L^{q'}$, $\vec{u}_{\restriction t=0}=(u_0,u_1)\in L^2(\Omega)\times \dot{H}^{-1}$ 
 one has 
$$\sup_{t\in \R}\|\PPP_c\vec{u}(t)\|_{L^2\times \dot{H}^{-1}}\leq C\left( \|P_c f\|_{L^{p'}L^{q'}}+\|\PPP_c(u_0,u_1)\|_{L^2\times \dot{H}^{-1}}\right).$$
 \end{theorem}
 \begin{proof}
 By Corollary \ref{cor:H-1_wave} we can assume $(u_0,u_1)=(0,0)$.
  The proof in this case is similar to Step 3 of the proof of Theorem \ref{T:wave}, and uses Theorem \ref{T:wave}, a duality argument and Christ and Kiselev Lemma. Let $\mathcal{A}_V$ be defined by \eqref{defAV}.
Let $T>0$.
By Theorem \ref{T:wave}, the operator $(u_0,u_1)\mapsto \Pi_0 e^{it\AAA_V}\PPP_c(u_0,u_1)$ is bounded from $\dot{H}^{1}_0\times L^2$ to $L^p(0,T),L^q(\Omega))$. By duality,
$$f\mapsto\int_0^T \PPP_c\left(e^{-is\AAA_V} (0,f)\right)ds$$
is bounded from $L^{p'}((0,T),L^{q'}(\Omega))$ to $L^2\times \dot{H}^{-1}$. By Corollary \ref{cor:H-1_wave},
$$f \mapsto \int_0^{T}\PPP_c\left(e^{i(t-s)\AAA_V} (0,f)\right)ds $$
is bounded from $L^{p'}((0,T),L^{q'}(\Omega))$ to $C^0\left([0,T],(L^2\times \dot{H}^{-1})(\Omega)\right)$. The conclusion follows from Christ and Kiselev Lemma, and the fact that the estimates are uniform in $T$.
 \end{proof}

By interpolation, combining Theorems \ref{T:wave} and Theorem \ref{T:Strichartz_H-1}, one obtains Strichartz estimates for solutions of equation (1.1) with initial data $\dot{H}^{\gamma}_D\times \dot{H}^{\gamma-1}_D$, $\gamma\in (0,1)$, where the spaces $\dot{H}^{\gamma}_D$ are defined as follows:
\begin{definition}
 For $0<\gamma<1$,  $\dot{H}^{\gamma}_D$ is the complex interpolate $\Big[L^2(\Omega),\dot{H}^1_0(\Omega)\Big]_{\gamma}$, and $\dot{H}^{-\gamma}_D$ is the complex interpolate $\left[L^2(\Omega),\dot{H}^{-1}(\Omega)\right]_{\gamma}$
\end{definition}
We next give a typical result:
\begin{corollary}
\label{cor:fractional}
Assuming that Assumptions \ref{As:K}, \ref{As:V} and \ref{As:zero} hold, and that Assumption \ref{As:pqwave} holds for all pairs $(p,q)$ with $p>2$ that satisfies \eqref{admissible_wave}, one has, for all solution of \eqref{Wave_intro}, \eqref{Wave_ID},
$$\|P_c u\|_{L^pL^q}\lesssim \left\|P_c u_0\right\|_{\dot{H}^{\gamma}_D}+\left\|P_c u_1\right\|_{\dot{H}^{\gamma-1}_D}+\|P_c f\|_{L^{r'}L^{s'}}$$
for all $\gamma\in [0,1]$, $(p,q)$, $(r,s)$ satisfying
\begin{equation}
 \label{cond1}
 \frac{1}{p}+\frac{N}{q}=\frac{N}{2}-\gamma=\frac{1}{r'}+\frac{N}{s'}-2
\end{equation} 
and
\begin{equation}
 \label{cond2}
 \frac{1}{p}+\frac{N-1}{2q}\leq \frac{N-1}{4},\quad
 \frac{1}{r}+\frac{N-1}{2s}\leq \frac{N-1}{4}
\end{equation}
with $p>2$, $r>2$.
\end{corollary}
For example, the conclusion of Corollary \ref{cor:fractional} is valid when  Assumptions \ref{As:K}, \ref{As:V} and \ref{As:zero} hold, and $K$ is strictly convex.
When Assumptions \ref{As:K}, \ref{As:V} and \ref{As:zero} hold and Assumption \ref{As:pqwave} is only valid for some of the admissible pairs $(p,q)$, it is of course possible to obtain Strichartz estimates as in Corollary \ref{cor:fractional} for a smaller set of parameters $(p,q,r,s,\gamma)$.
\begin{proof}[Proof of Corollary \ref{cor:fractional}]
 The case $\gamma=1$ is given by Theorem \ref{T:wave} and the case $\gamma=0$ by Theorem \ref{T:Strichartz_H-1}. The case $\gamma\in (0,1)$ follows by interpolation, using Theorem 5.1.2 of \cite{BerghLofstrom76BO}. Note that in the notations of \cite{BerghLofstrom76BO}, introduced before Theorem 5.1.2, the space  $L_{\infty}([-T,T],dt,A)$ (where $A$ is a Banach space) is spanned by simple functions, so that the spaces $L_{\infty}^0(A)$ and $L_{\infty}(A)$ coincide in this case. The interpolation procedure can thus be carried out on $[-T,T]$ for all $T>0$, which yields the result.
\end{proof}

\end{document}